\RequirePackage{fix-cm}
\documentclass[smallextended]{svjour3}      
\smartqed  
\usepackage[referable]{threeparttablex}
\usepackage{footnote}
\usepackage{mathtools}
\usepackage{lipsum}
\usepackage{hyperref}
\usepackage{graphicx}
\usepackage{hyperref}
\hypersetup{urlcolor=blue, citecolor=red, colorlinks=true}
\usepackage{fullpage}
\usepackage{amsmath,amssymb,floatflt,amsfonts,amscd,graphics} 
\usepackage{verbatim}
\usepackage{algorithm,algorithmic}
\makeatletter
\renewcommand\fs@ruled{%
  \def\@fs@cfont{\rmfamily}%
  \let\@fs@capt\floatc@plain%
  \def\@fs@pre{\hrule height.8pt depth0pt \kern2pt}
  \def\@fs@post{\kern2pt\hrule\relax}
  \def\@fs@mid{\kern2pt\hrule\kern2pt}
  \let\@fs@iftopcapt\iffalse}
\makeatother
\newtheorem{assumption}{Assumption}
\newcommand*{\QEDB}{\hfill\ensuremath{\square}}%
\newcommand{\Euclidsp}{\mathbb{E}}
\newcommand{\iprod}[2]{\left\langle {#1}, {#2} \right\rangle}

\newcommand{\lrpar}[1]{\left(#1\right)}

\newcommand{\proj}{\Pi}
\newcommand{\argmin}{\arg\min}
\newcommand{\norm}[1]{\left\lVert {#1} \right\rVert}
\newcommand{\bracket}[1]{\left( {#1} \right)}
\newcommand{\nat}{\mathrm{nat}}

\newcommand{\bbE}{\mathbb{E}}

\newcommand{\dist}{\mathrm{dist}}
\newcommand{\natmap}{G^\nat}

\newcommand{\der}{\text{\textbf{J}}}

\newcommand{\derD}{\mathrm{D}}

\newcommand{\sgn}{\mathrm{sgn}}
\newcommand{\spn}{\mathrm{span}}

\newcommand{\logdet}{\mathrm{logdet}}
\newcommand{\Diag}{\mathrm{Diag}}

\newcommand{\trace}{\mathrm{Tr}}
\newcommand{\relint}{\mathrm{relint}}
\newcommand{\inter}{\mathrm{int}}
\newcommand{\coni}{\mathrm{cone}}
\newcommand{\affine}{\mathrm{aff}}
\newcommand{\lrbrace}[1]{\left\{#1\right\}}

\begin{document}

\title{A global linear and local superlinear (quadratic) inexact non-interior continuation method for variational inequalities over general closed convex sets}



\author{Le Thi Khanh Hien        \and Chek Beng Chua
}


\institute{L. T. K. Hien \at
              Department of Mathematics and Operations Research, University of Mons, Belgium. \\           
              \email{thikhanhhien.le@umons.ac.be}
           \and
      C. B. Chua \at
              School of Physical \& Mathematical Sciences, Nanyang Technological University \\           
              \email{cbchua@ntu.edu.sg}                        
}

\date{Received: date / Accepted: date}
\maketitle





\begin{abstract}
We use the concept of barrier-based smoothing approximations to extend the non-interior continuation method, which was proposed by B. Chen and N. Xiu for nonlinear complementarity problems based on Chen-Mangasarian smoothing functions, to an inexact non-interior continuation method for variational inequalities over general closed convex sets. Newton equations involved in the method are solved inexactly to deal with high dimension problems. The method is proved to have global linear and local superlinear/quadratic convergence under suitable assumptions. We apply the method to non-negative orthants, positive semidefinite cones, polyhedral sets, epigraphs of matrix operator norm cone and epigraphs of matrix nuclear norm cone. 
\end{abstract}


\section{Introduction}
Let $X$ be a given closed convex subset of a finite dimensional real vector space $\mathbb{E}$ equipped with an inner product $\langle\cdot,\cdot\rangle$,  and  $F: \mathbb{E} \rightarrow \mathbb{E}$ be a continuously differentiable map.  We consider the following variational inequality $VI(X,F)$ over general closed convex sets: find $x\in X$ such that
\begin{equation}\label{eq:compl_equation}
 \langle F(x), y-x \rangle \geq 0 \; \mbox{for all} \; y\in X.
\end{equation}
Many well-known optimization problems can be cast as VIs. For examples, when $X$ is a cone in $\mathbb{E}$, the variational inequality (VI) is equivalent to a nonlinear complementarity problem (NCP) of finding $x \in X$ such that $F(x) \in X^\sharp$ and $\iprod{x}{F(x)} =0$, where $X^\sharp$ is the dual cone of $X$; convex optimization problems and fixed point problems can also be cast as VIs, see \cite[Chapter 1]{Facchinei}.  We let $\norm{\cdot}$ be the norm induced by the inner product $\iprod{\cdot}{\cdot}$, and denote the Euclidean projection of $z$ onto $X$ by $\proj_X(z)$, i.e.,
\[
\proj_X(z) = \argmin_{x \in X} \frac12 \norm{x - z}^2.
\]
The map \[
(x, y) \in \mathbb{E} \times \Euclidsp \mapsto (x - \proj_X(x-y), F(x) - y)
\]
is called the \emph{natural map}  $\natmap$. It was proved in \cite[Proposition 1.5.8]{Facchinei} that $x$ is a solution of $VI(X,F)$ if and only if $x$ satisfies $\natmap (x,y)=0$ for some $y\in \mathbb{E}$, and this equation is called the \emph{natural map equation}. In this paper, we are interested in solving $VI(X,F)$ via the natural map equation.

The natural map equation, in general, is nonsmooth since the Euclidean projection is not always smooth; and as such it can not be solved by typical Newton-based methods. To remedy this restriction, one can use nonsmooth Newton-based methods, see e.g., \cite[Chapter 7, Chapter 8]{Facchinei2}, \cite{Pang2,Pang,QiSun,Ralph}, or use smoothing approximations of the Euclidean projection to solve the natural map equation numerically by a smoothing Newton continuation method. A continuous map $p:\mathbb{E} \times \mathbb{R}_{+} \rightarrow \mathbb{E}$, parameterized by $\mu\in\mathbb{R}_{+}$, is called a smoothing approximation (SA) of the Euclidean projection $\Pi_X$ over a closed convex set $X$ if $p$ converges point-wise to $\Pi_X$ as $\mu \rightarrow 0$, i.e., $ p(z,0)=\Pi_X(z)$, and for each $\mu\in \mathbb{R}_{++}$, $p(\cdot,\mu)$ is differentiable. 
When the convergence is uniform, $p$ is called a uniform SA.
A very well-known example of the SA is the Chen-Harker-Kanzow-Smale (CHKS)
function (see \cite{Chen_Harker,Kanzow2}) $p(z,\mu)=\frac12(\sqrt{z^2 + 4\mu^2} + z)$, which approximates the projection  onto the set of non-negative real numbers $\mathbb{R}_+$. The CHKS function belongs to the class of smoothing functions introduced by C. Chen and O. L. Mangasarian, see \cite{Chen_Mangasarian}, which is computed as 
$$p(z,\mu)=\int_{-\infty}^z\int_{-\infty}^t \frac{1}{\mu} d\bracket{\frac{x}{\mu}} dx dt,
$$
where $d(\cdot)$ is a certain probability density function. When $d(x)=\frac{2}{(x^2+4)^{3/2}}$ the double integral equals to the CHKS function. L. Qi and D. Sun \cite{QiSun02} developed this type of convolution-based SAs to  approximate general nonsmooth functions. However, it is not computable in most cases since it contains a multivariate integral. Recently, C. B. Chua and Z. Li \cite{Chua_Li} introduced barrier-based smoothing functions, which only approximate the Euclidean projection onto convex cone with nonempty interior. This type of SAs has been extended to general closed convex sets with non-empty interior in \cite{Chua_Hien}. Note that the barrier-based SA can be computed via proximal mappings of smooth maximal monotone maps. 
We denote $p_\mu(z)=p(z,\mu)$ to emphasize that $\mu$ will be used as a parameter, and define a SA of the natural map $\natmap$
\begin{align} \label{eq:natmap}
H_\mu(x,y)=\begin{pmatrix}
x-p_\mu(x-y)\\
F(x) - y \end{pmatrix} .
\end{align} 

It is worth noting that, in the literature, there are many approaches to solve the VI in its general form \eqref{eq:compl_equation} or when $X$ or $F$ have more specific structures, for examples, KKT conditions based methods, merit function based algorithms, interior point methods, projection methods, to name a few. We refer the readers to \cite{Facchinei2} for a comprehensive review of algorithms for solving VIs. In this paper, we are specifically interested in the SA approach which we briefly review in the next paragraph. 

Chen and Mangasarian are the pioneers in using smoothing methods for solving VIs. Based on the smoothing of the plus function, Newton-based algorithms are proposed in \cite{Chen_Mangasarian} for solving NCPs and box-constrained VIs. The authors in \cite{Gabriel1995Smoothing} proposed a class of smoothing functions and use it to approximate the mixed NCP by a smooth system of nonlinear equations. They then study the limiting behaviour of the path generated by the approximate solutions of a sequence of least squares problems. Chen et al in  \cite{ChenQiSun1998} define an important property for the sequence of the derivatives of the SA which is called the  Jacobian consistency property; by using this property together with some mild assumptions, for the first time in the literature, a local superlinear convergence rate is established for a smoothing Newton method that solves \emph{a nonsmooth equation}. Note that the global convergence rate is not established for this smoothing Newton method.  The method is then applied to solve a box constrained VI.  Li and Fukushima \cite{Li2000} derive the CHKS smoothing function (which satisfies the Jacobian consistency property) for the mixed complementarity problem; and under suitable conditions, they establish local quadratic convergence properties for a smoothing Newton method. In another series of works, the SA approach is incorporated with  the idea of path-following methods to yield the non-interior continuation methods, which are also known as non-interior path-following methods, for solving NCPs, see \cite{Chen_Harker,Chen1995}. Burke and Xu \cite{Burke_Xu}, for the first time, establish the global linear convergence of the non-interior continuation method for linear complementarity problem, and extended the result to NCPs with uniform $P$-functions \cite{Xu2000}. The method are further extended and well-documented in the literature, see e.g.,  \cite{Bintong,Chen_Ye,Kanzow2} and references therein. 

We notice that although SA algorithms have been deeply studied for NCPs and VIs with specific structures of $X$, studying the algorithms for solving the VI over a general convex set $X$ is still an interesting and challenging topic. The main difficulty of designing smoothing approximation algorithms for VIs over general convex sets lies in how we develop the smoothing approximation of the Euclidean projection and the properties of its derivative sequence (such as the Jacobian consistency property). To the best of our knowledge, there exists only two approaches of smoothing approximation of Euclidean projections onto general convex sets, which are mentioned above -- the convolution-based and the barrier-based approximation. While the former approximation has been widely applied and employed to develop smoothing approximation algorithms for VIs with strong convergence rate guaranty, the later still has potential developments to explore. In this paper, we will employ the barrier-based smoothing functions to approximately solve the natural map equation.

Solving smooth equations by classical Newton method was proved to have local quadratic convergence. However, it was also verified that if a large size problem is considered then solving a system of Newton equations at each iterate would be very expensive. Inexact Newton methods \cite{Dembo}, which calculate an appropriate solution to the Newton equations satisfying some level of precision, are more practical and suitable for large scale problems. In particular, to solve the system of nonlinear equation $G(x)=0$, instead of using the Newton's step $x_{k+1}=x_k + s_k$, where $s_k$ is calculated by $G'(x_k) s_k=-G(x_k)$, the authors in \cite{Dembo} propose a class of inexact Newton methods which solves the Newton equation inexactly $G'(x_k) s_k=-G(x_k) + r_k$, where the relative residual  $\norm{r_k}/\norm{G(x_k)}$ does not exceed some predetermined precision $\eta_k$. Many Newton-like methods that do not require to compute the closed forms of $G'(x_k)$ or its inverse belong to this class of inexact Newton methods. For example, Dennis \cite{Dennis1968} proposes a Newton-like method that updates $x_{k+1}=x_k - M(x_k) G(x_k)$, where $M(x_k)$ is a linear operator satisfying some conditions; under these conditions we can prove that the relative residual, with $r_k=G'(x_k) (-M(x_k)G(x_k))+G(x_k)$, satisfies some precision, see \cite[Section 4]{Dembo}. Some other examples include Newton-Krylov methods \cite{Brown1990,KNOLL2004357}, truncated Newton methods \cite{Pernice1998}. The inexact Newton step $s_k$ can also be found by applying efficient iteration solvers for the system of linear equations such as the GMRES method \cite{Bellavia2001,Saad1986}, the splitting methods \cite{Bai2003,Bai2005}, conjugate gradient methods \cite{Stoer1983}. The inexact Newton methods also converge locally superlinearly/quadratically under some natural assumptions of the relative residuals, see \cite{Dembo}. Taking into account these advantages, we use inexact Newton methods to solve Newton equations in our path-following continuation method.

\subsection{Contribution}
Our main contribution is that, using the barrier-based smoothing approximation, we improve the noninterior continuation method in \cite{Bintong} to solve \emph{VI over general closed convex sets}. The method employs centering steps, which give its \emph{global linear convergence}, together with approximate Newton steps, which help to achieve  \emph{local superlinear/quadratic convergence}. Newton equations involved in the method are solved inexactly to handle large scale problems. We provide the application of our method to non-negative orthants, positive semidefinite cones, polyhedral sets, epigraphs of matrix operator norm and epigraphs of matrix nuclear norm. 

To achieve global linear convergence rate, we assume that the derivative $\derD H_{\mu_k}(x^{(k)},y^{(k)})$, where $\left\{(x^{(k)},y^{(k)})\right\}_{k\geq 0}$ is the sequence generated by our algorithm, is nonsingular and $\left\{\norm{\derD H_{\mu_k}(x^{(k)},y^{(k)})^{-1}}\right\}_{k\geq 0}$ is bounded. We prove that the monotonicity of $F$ is sufficient for  the non-singularity of $\derD H_{\mu_k}\bracket{x^{(k)},y^{(k)}}$. It was proved in \cite{Chen_Tseng} that, when $X$ is a positive semidefinite cone, the strong monotonicity  of $F$ together with the uniform boundedness of $\lrbrace{\norm{\derD F(x^{(k)})}}_{k\geq 0}$ is sufficient for the boundedness of the sequence $\left\{\norm{\derD H_{\mu_k}(x^{(k)},y^{(k)})^{-1}}\right\}_{k\geq 0}$. We will extend this result to the case when $X$ is a general convex set, see Section \ref{subsec:sufficient conds}. It is worth mentioning that, for more specific cases of $X$, the strong monotonicity condition of $F$ can be relaxed. We refer the readers to \cite{Bintong} for the case when $X$ is a non-negative orthant, to \cite{Chua_Hien} for the case when $X$ is an epigraph of operator norm or an epigraph of nuclear norm.

To obtain local superlinear convergence rate, we develop an important property of the derivative sequence $\derD p_\mu(z)$ for the barrier-based SA. In particular, we prove that if $\derD p_\mu(z)$ converges to a linear operator $T^*$ when $(z,\mu)$ converges to $(z^*,0)$, where $z^*=x^*-y^*$ and $(x^*,y^*)$ is any limit point of $\left\{(x^{(k)},y^{(k)})\right\}_{k\geq 0}$, then the operator $T^*$ must be $\derD \Pi_X(z^*)$, which also implies that the projector onto $X$ is differentiable at $z^*$. To achieve $\xi$-order convergence rate with $\xi>1$, we further need  $ \norm{\derD p_\mu(z)-T^*}=O\bracket{\norm{(z-z^*,\mu)}^{\xi-1}}$. We verify the property with $\xi=2$ (i.e., the local convergence rate will be quadratic) for non-negative orthants, positive semidefinite cones, polyhedral sets, epigraphs of matrix operator norm and epigraphs of matrix nuclear norm. We further show that for non-negative orthants, positive semidefinite cones, epigraphs of matrix operator norm and epigraphs of matrix nuclear norm, differentiability of $\Pi_X(\cdot)$ at $z^*$ is equivalent to strict complementarity of $(x^*,y^*)$.   

\subsection{Organization and notation}
The paper is organized as follows. In the next section, we give some preliminaries on barrier-based smoothing approximation. In Section \ref{sec:method_describe}, we describe the inexact non-interior continuation method and prove its global linear, and local superlinear/quadratic convergence. We present application of the algorithm to specific convex sets in Section \ref{sec:application}. Finally, we conclude our paper in Section \ref{sec:end}.

We end this section by explaining some notations that will be used in the paper. For a  Fr{\'e}chet-differentiable map $F$, we use $\derD F$ to denote the derivative of $F$ and $\der F$ to denote its Jacobian. \color{black}
We use  $\nabla f(x)$ and  $\nabla^2 f(x)$ to denote the gradient and the Hessian of a twice Fr{\'e}chet-differentiable function $f$. For a vector $x\in \mathbb{R}^m$, $[x]_+$ denotes the vector whose components are $[x_i]_+=\max\{0,x_i\}$, $i=1,\ldots,m$, and $\Diag(x)$ denotes the $m\times m$ diagonal matrix with $(\Diag(x))_{ii}= x_i, i=1,\ldots,m$. For a matrix $z\in \mathbb{R}^{m\times m}$ we define two linear matrix operators
$$ \mathfrak S(z)=\cfrac12(z+z^T), \quad \mathfrak T(z)=\cfrac12(z-z^T).
$$
We denote the following norms of a matrix $z\in \mathbb{R}^{m\times n}$
\begin{itemize}
\item[$\bullet$] $\|z\|_\infty$: the $l_\infty$ norm, i.e, $\|z\|_\infty=\max\{ |z|_{ij}, 1\leq i\leq m, 1\leq j\leq n\},$
\item[$\bullet$] $\|z\|_1$: the $l_1$ norm, i.e, $\|z\|_1=\sum\limits_{ 1\leq i\leq m, 1\leq j\leq n}|z_{ij}|,$
\item[$\bullet$] $\norm{z}_F$: the Frobenius norm,
\item[$\bullet$] $\|z\|$: the operator norm, i.e, the largest singular value of $z$,
\item[$\bullet$] $\|z\|_*$: the nuclear norm, i.e, the sum of all singular values of $z$.
\end{itemize}
We use $I$ to denote the identity matrix, whose dimension is clear in the context, and $\mathbf{I}$ to denote the identity operator. For operators $P_1, P_2:\mathbb{E}\to \mathbb{E}'$, we use $P_1\prec P_2$ to mean $\iprod{(P_2-P_1)[u]}{u}>0$ for all $u\ne 0$. We denote the set of $n\times n$ symmetric matrices by $\mathbb{S}^{n}$, the set of $n\times n$ orthogonal matrices by $\mathcal{O}^n$. For $z\in \mathbb{S}^n$, we use $\lambda_1(z) \geq \lambda_2(z) \geq \lambda_m (z)$ to denote its eigenvalues with multiplicity, $\lambda_f(z)=(\lambda_1(z),\ldots,\lambda_m(z))^T$, and $\mathcal{O}(z)$ to denote the set of orthogonal matrices $u$ such that $z=u \Diag(\lambda_f(z)) u^T$. For a given matrix $z\in \mathbb{R}^{m\times n}$ with $m\leq n$, we denote $\sigma_1(z) \geq \sigma_2(z) \geq \ldots \geq \sigma_m(z)$ to be singular values of $z$ with multiplicity and $\sigma_f(z)=(\sigma_1(z),\ldots,\sigma_m(z))^T$. We use $\circ$ to denote the Hadamard product (or the entry-wise product), i.e., $(z\circ w)_{ij}=z_{ij} w_{ij}$ for $z,w\in \mathbb{R}^{m\times n}$. For a $n\times n$ matrix $z$ we use $\trace(z)$ to denote its trace. The norm of a multilinear operator $L: \mathbb{E}\times\ldots\times\mathbb{E} \rightarrow \mathbb{E'}$ is given by
$$\norm{L}=\sup\limits_{v=(v_1,\ldots,v_k)\in \mathbb{E}\times\ldots\times \mathbb{E}} \left\{ \cfrac{\norm{Lv}}{\norm{v}}: v\ne 0\right\},
$$
where the norm of $v=(v_1,\ldots,v_k)\in \mathbb{E}\times\ldots\times \mathbb{E}$ is $\norm{v}=\sqrt{\norm{v_1}^2+\ldots+\norm{v_k}^2}$. We use $\inter(X)$ and $\relint(X)$ to denote the interior and the relative interior of $X$. 
\section{Preliminaries}
\label{sec:Prel}
In this section, we give some background knowledge and establish some necessary results to be used in the sequel.
\begin{definition}\label{def:barrier}
(a) A function $f:\inter(X) \rightarrow \mathbb{R}$ is called a barrier of $X$ if $f(x_k)\rightarrow \infty$ for any sequence $\{x_k\} \subset \inter(X)$ converging to a boundary point of $X$.

(b) A function $f$ is called $\vartheta$-self-concordant barrier for $X$ if it is three times continuously differentiable, and satisfies the following conditions
\begin{itemize}
\item[($\alpha$)]  the following value, called barrier parameter, is finite
$$\vartheta=\inf \left\{ t\geq 0:  \inf\limits_{x\in \inter(X), h\in \mathbb{E}} t \langle h, \nabla^2 f(x)h \rangle -\langle \nabla f(x),h\rangle^2 \geq 0\right\};
$$
\item[($\beta$)] $|\derD^3 f(x)[h,h,h]| \leq 2 (\derD^2 f(x)[h,h])^{3/2} \quad \forall x\in \inter(X), h\in \mathbb{E}$.
\end{itemize}
\end{definition}
\begin{proposition}\label{prop:barrier_property} 
If $f$ is a $\vartheta$-self-concordant barrier of $X$ then
$$| \derD^3 f(x)[h_1,h_2,h_3]| \leq 2 \bracket{\derD^2 f(x)[h_1,h_1] \derD^2 f(x)[h_2,h_2] \derD^2 f(x)[h_3,h_3]}^{1/2}.
$$
\end{proposition}
We refer to \cite[Appendix 1]{Yurii} for the proof.

For a given closed convex set $X$ with a differentiable barrier $f$, we define the map $p:\mathbb{E}\times \mathbb{R}_+ \rightarrow \mathbb{E}$ as
\begin{align}\label{eq:bsa} \left\{
\begin{array}{ll}
p(z,\mu) + \mu^2 \nabla f(p(z,\mu)) =z\quad &\mbox{when} \;\mu>0
\\ p(z,\mu)= \Pi_X(z)\quad &\mbox{when} \; \mu=0.
\end{array}\right.
\end{align}
Condition ($\alpha$) in Definition \ref{def:barrier} is used to prove the following theorem, which leads to Definition \ref{def:smoothing} -- the definition of barrier-based smoothing approximation.
\begin{theorem}\cite[Theorem 3.1, Theorem 3.2]{Chua_Hien}\label{theorem:bsa}
If $f$ is a twice continuously differentiable barrier on $X$, then the map $p$ defined via \eqref{eq:bsa} is a smoothing approximation of the Euclidean projector $\Pi_X$. In addition, if $f$ is a $\vartheta$-barrier then the map $p$ is Lipschitz continuous with modulus $\sqrt{\vartheta}$ in the smoothing parameter; consequently, $p$ is a uniform smoothing approximation.
\end{theorem}
\begin{definition} 
\label{def:smoothing}
The barrier based smoothing approximation of Euclidean projection $\Pi_X$ defined by a given twice continuously differentiable barrier $f$ on $X$ is the map $p:\mathbb{E} \times\mathbb{R} \rightarrow \mathbb{E}$ that satisfies \eqref{eq:bsa}.
\end{definition}
We remind that $p_\mu(z)=p(z,\mu)$. The following proposition provides an upper bound for  $\norm{ \derD^2 p_\mu(z)}$ which is of paramount importance in achieving the global linear convergence of our inexact non-interior continuation method  proposed in Section \ref{sec:method_describe}.
\begin{proposition} \label{prop:bsa_property}
 If the barrier $f$ of $X$ is $\vartheta$-self-concordant, then for all $z\in \mathbb{E}$ and $\mu>0$ we have   
 $$\norm{ \derD^2 p_\mu(z)} \leq \frac{1}{4\mu}.$$
\end{proposition}
\begin{proof}
Denote $g(z)=\nabla f(z)$. From the definition of barrier-based smoothing approximation, we have
$$ p_\mu(z)+ \mu^2g(p_\mu(z))=z.
$$
Taking derivative on $z$ both sides, we get
\begin{equation} 
\label{eq:derofp}
\derD p_\mu(z)[u] + \mu^2\derD g(p_\mu(z))[\derD p_\mu(z)[u]]=u,
\end{equation}
for all $u\in \mathbb{E}$. 
 Taking derivative again on both sides of \eqref{eq:derofp} gives us
\begin{align*}
 \derD^2 p_\mu(z)[u,v] + \mu^2\derD^2 g(p_\mu(z))[\derD p_\mu(z)u, \derD p_\mu(z)v]  + \mu^2 \derD g(p_\mu(z))\derD^2 p_\mu(z)[u,v]=0,
 \end{align*}
 for all $u,v\in \mathbb{E}$.
Therefore, we get 
\begin{equation}
\label{eq:normrho}\norm{\derD^2 p_\mu(z)[u,v]} = \norm{\mu^2\derD p_\mu(z)  \derD^2 g(p_\mu(z))\big[ \derD p_\mu(z)u,\derD p_\mu(z)v]}.
\end{equation}
Denote $\rho=\mu^2 \derD p_\mu(z)  \derD^2 g(p_\mu(z))\big[ \derD p_\mu(z)u,\derD p_\mu(z)v]$. Now we prove that $\norm{\rho}\leq \frac{1}{4\mu}$.  We remind that $g=\nabla f$. For all $w\in\mathbb{E}$ with $\norm{w}=1$, we have 
\[
\iprod{\rho}{w}=\mu^2\derD^3 f(p_\mu(z))\big[ \der p_\mu(z)u\;,\;\der p_\mu(z)v\;,\;\der p_\mu(z)w\big].
\]
Since $f$ is a $\vartheta$-self-concordant barrier, we get
\begin{equation} 
\label{temp0}
\begin{split}
&\mu^2\derD^3 f(p_\mu(z))\big[ \der p_\mu(z)u\;,\;\der p_\mu(z)v\;,\;\der p_\mu(z)w\big]
\\
&\leq 2\mu^2\prod\limits_{d\in\{u,v,w\}} \Big| \derD^2 f(p_\mu(z)) \big[\der p_\mu(z)d ,  \der p_\mu(z) d \big]\Big|^{1/2}.
\end{split}
\end{equation}
 Let $\nabla^2 f(p_\mu(z))= Q\; \Diag\left(\lambda_i|_{1\leq i\leq n} \right) Q^T$ be eigenvalue decomposition of $\nabla^2 f(p_\mu(z))$, whose eigenvalues are $\lambda_i, 1\leq i \leq n$. From \eqref{eq:derofp}, we have
$\der p_\mu(z)= Q \Diag \left( \left.\cfrac{1}{1+ \mu^2\lambda_i}\right|_{1\leq i \leq n}\right) Q^T.$
By noting that $\cfrac{\mu^2\lambda_i}{(1+ \mu^2\lambda_i)^2} \leq \cfrac14$, we imply that for $d\in\{u,v,w\}$ the following satisfies
\begin{align*}
\mu^2\Big| \derD^2 f(p_\mu(z)) \big[\der p_\mu(z)d,\der p_\mu(z) d \big]\Big| &= d^T Q \Diag \left(\left.\cfrac{\mu^2\lambda_i}{(1+ \mu^2\lambda_i)^2}\right|_{1\leq i \leq n}\right)Q^T d \\
&\leq \frac14\norm{d}^2\norm{Q}^2\leq \frac14.
\end{align*}
Together with \eqref{temp0}, we get
$\iprod{\rho}{w}\leq \cfrac{1}{4\mu}, \forall\; w\in\mathbb{E}, \norm{w}=1.$ Hence $\norm{\rho}=\max_{w:\norm{w}=1}\iprod{\rho}{w}\leq \frac{1}{4\mu}.$
We deduce from \eqref{eq:normrho} that for all $u,v\in\mathbb{E}$ with $\norm{(u,v)}=1$ we have $\norm{\derD^2 p_\mu(z)[u,v]}\leq \frac{1}{4\mu}$.  

~\QEDB
\end{proof} 
The following theorem connects some properties of the projection $\Pi_X(\cdot)$ with the behaviour of $\derD p_\mu(z)$ at its limit point (which is supposed to exist). Theorem \ref{theorem:seq_Jacob_conver} serves as a cornerstone to prove the local convergence of our algorithm. 
\begin{theorem} \label{theorem:seq_Jacob_conver}
If $ \lim\limits_{(z,\mu)\rightarrow (z^*,0)} \derD p_\mu(z)=T^*$, then the projector $\Pi_X$ is strictly differentiable at $z^*$ and  $\derD \Pi_X(z^*) = T^*$. Furthermore, if $\norm{\derD p_\mu(z)-T^*} =O(\|(z-z^*,\mu)\|^{\xi-1})$ with $\xi>1$ then $\Pi_X$  satisfies
$$\norm{\Pi_X(z) - \derD \Pi_X(z^*)- \Pi_X(z^*)} = O(\norm{z-z^*}^\xi).
$$
\end{theorem}
\begin{proof}
Since $ \lim\limits_{(z,\mu)\rightarrow (z^*,0)} \derD p_\mu(z)=T^*$, we have 
$$\forall \varepsilon>0,\; \exists \delta>0\; \text{such that} \, 0< \norm{ (z-z^*, \mu)} < \delta \Rightarrow \norm{\derD p_\mu(z) - T^*}< \varepsilon.$$
Hence, $\forall \varepsilon>0,\; \exists \delta>0,$ for any $z_1, z_2$ and $\mu$ such that $0<\norm{z_1-z^*} < \cfrac{\delta}{\sqrt2}$, $0<\norm{z_2-z^*} < \cfrac{\delta}{\sqrt2}$ and $\mu  < \cfrac{\delta}{\sqrt2}$, the following holds 
\begin{equation} \label{temp00}
\begin{split}
\norm{p_\mu(z_2) - p_\mu(z_1) - T^*[z_2-z_1]} &= \norm{\int\limits_{0}^1 (\derD p_\mu(t z_2 + (1-t)z_1) - T^*)dt [z_2-z_1]}\\
&< \varepsilon \norm{z_2-z_1},
\end{split}
\end{equation}
since $\norm{t z_2 + (1-t)z_1 - z^* } \leq t \norm{z_2-z^*} + (1-t)\norm{z_1-z^*} < \cfrac{\delta}{\sqrt2}$ for all $0\leq t\leq 1$. 
Moreover, if $\norm{\derD p_\mu(z)-T^*} =O(\|(z-z^*,\mu)\|^{\xi-1})$ then  $\forall \varepsilon>0, \exists \delta>0,$ for  $0<\norm{z-z^*} < \cfrac{\delta}{\sqrt2}$ and $\mu  < \cfrac{\delta}{\sqrt2}$, we get
\begin{equation} \label{temp01}
\begin{split}
\norm{p_\mu(z^*+h) - p_\mu(z^*) - T^*[h]} & = \norm{\int\limits_{0}^1 (\derD p_\mu(z^* + th) - T^*)dt [h]}\\
&= O(\|h\|^\xi) + O(\mu^{\xi-1})\|h\|.
\end{split}
\end{equation} 
As $\mu$ in \eqref{temp00} goes to 0, accompanied with $p_\mu(z_2)\rightarrow \Pi_X(z_2)$,  $p_\mu(z_1)\rightarrow \Pi_X(z_1)$, it yields that
$$\norm{\Pi_X(z_2) - \Pi_X(z_1) - T^*[z_2-z_1]} \leq \varepsilon \|z_2-z_1\|.
$$
This expression shows that the projection $\Pi_X$ is strictly differentiable at $z^*$ and $\derD \Pi_X(z^*) = T^*$. Similarly, as $\mu$ in \eqref{temp01} goes to 0, the second result follows
$$\norm{\Pi_X(z^*+h) - \Pi_X(z^*) - T^*[h]} = O(\|h\|^\xi). 
$$ \QEDB
\end{proof}
\begin{remark}
The inverse direction of Theorem \ref{theorem:seq_Jacob_conver} does not always hold. If we use a wrong barrier function for the set $X$, the limit may not exist even when $\Pi_X$ is differentiable at $z^*$; see an example in  \ref{example} . However, the question about the existence of a suitable barrier-based smoothing approximation which may guarantee the limit for a given general convex set $X$ is still open to us.  
\end{remark}
\section{An inexact non-interior continuation method} \label{sec:method_describe}
In this section, we describe the inexact non-interior continuation for solving Problem $\eqref{eq:compl_equation}$ and prove its local superlinear ($\xi$- order) and global linear convergence. We use a $\vartheta$-self-concordant barrier $f$  to formulate $p_\mu(\cdot)$. Denote $\phi_\mu(x,y) = x- p_\mu(x-y)$ and remind that $H_{\mu}(x,y)=\begin{pmatrix}
x-p_\mu(x-y)\\
F(x) - y
\end{pmatrix}.$
 We use the merit function 
$$\Psi_\mu(x,y)=\| F(x)-y\| + \| \phi_\mu(x,y)\|,$$ and define the neighbourhood
$\mathcal{N}(\beta,\mu)=\{ (x,y): \Psi_{\mu}(x,y) \leq \beta\mu\}.$
Algorithm \ref{alg:IeNICM} fully describes our algorithm.  

\begin{algorithm}[H] 
\caption{Inexact non-interior continuation method}
\label{alg:IeNICM}
\begin{algorithmic}
\STATE Given $\sigma,\alpha_1, \alpha_2,\alpha_3 \in (0,1)$, two sequences $\left\{\theta_1^{(k)}\right\}_{k\geq 0}\subset (0,1)$ and $\left\{\theta_2^{(k)}\right\}_{k\geq 0}\subset (0,1)$ such that $\sigma+\sqrt2  \sup_{k}\left\{\theta_1^{(k)}\right\} < 1$. Denote $w^{(k)}=(x^{(k)},y^{(k)})$.
\STATE \textbf{ Step 0} Set $k=0$. Choose $\mu_0 >0, w^{(0)} \in \mathbb{E}\times \mathbb{E}, \beta > \sqrt{\vartheta}$ such that $w^{(0)} \in \mathcal{N}(\beta,\mu_0)$.
\STATE \textbf{ Step 1} (Calculate centering step)
\IF{$H_0(w^{(k)})=0$} 
\STATE terminate, $w^{(k)}$ is a solution of the VI;
\ELSIF{$\Psi_{\mu_k}(w^{(k)})=0$} \STATE  set $\tilde{w}^{(k+1)}=w^{(k)}$ and go to step 3;
\ELSE \STATE let $\triangle \tilde{w}^{(k)}$ solve the equation
\begin{align} \label{eq:Newton_direction}
-H_{\mu_k}(w^{(k)})+r_1^{(k)}= \derD H_{\mu_k}(w^{(k)})[\triangle \tilde{w}^{(k)}],
\end{align}
where $\norm{r_1^{(k)}}\leq \theta_1^{(k)} \norm{H_{\mu_k}(w^{(k)})}$.
\ENDIF
\STATE \textbf{ Step 2} (Line search for centering step)
Let $\lambda_k$ be the maximum of $1, \alpha_1,\alpha_1^2,\ldots$ such that
\begin{align} \label{ieq1:line_search}
\Psi_{\mu_k}(w^{(k)} + \lambda_k \triangle \tilde{w}^{(k)}) \leq (1-\sigma \lambda_k) \Psi_{\mu_k}(w^{(k)}).
\end{align}
\STATE Set $\tilde{w}^{(k+1)}=w^{(k)} + \lambda_k \triangle \tilde{w}^{(k)}.$
\STATE \textbf{ Step 3} ($\mu$ Reduction based on centering step)
Let $\gamma_k$ be the maximum of the values $1, \alpha_2, \alpha_2^2, \ldots$ such that
$$\tilde{w}^{(k+1)} \in \mathcal{N}(\beta,(1-\gamma_k)\mu_k).$$
\STATE Set $\tilde{\mu}_{k+1} = (1-\gamma_k) \mu_k$.
\STATE \textbf{ Step 4} (Calculate approximate Newton step)
 Let $\triangle \hat{w}^{(k)}$ solve the equation
\begin{align} \label{eq:pure_Newton_direction}
-H_0(w^{(k)}) + r_2^{(k)}= \derD H_{\mu_k}(w^{(k)})[\triangle \hat{w}^{(k)}],
\end{align}
where $\norm{r_2^{(k)}} \leq \theta_2^{(k)} \norm{H_0(w^{(k)})}$.  
\STATE Set $\hat{w}^{(k+1)} = w^{(k)} + \triangle\hat{w}^{(k)}.$
\STATE \textbf{ Step 5} ($\mu$ reduction based on approximate Newton step)
\IF{$\hat{w}^{(k+1)} \not\in \mathcal{N}(\beta,\tilde{\mu}_{k+1})$} \STATE set
$\mu_{k+1}=\tilde{\mu}_{k+1}, w^{(k+1)} = \tilde{w}^{(k+1)}, k=k+1$
and return to step 1; 
\ELSIF{$H_0(\hat{w}^{(k+1)})=0$}
\STATE terminate, $\hat{w}^{(k+1)}$  is a solution of the VI; 
\ELSE \STATE let $\eta_k$ be the greatest value of $1, \alpha_3, \alpha_3^2, \ldots$ such that
$\hat{w}^{(k+1)} \not\in \mathcal{N}(\beta,\eta_k\alpha_3 \tilde{\mu}_{k+1}).$
\STATE Set
$\mu_{k+1}=\eta_k \tilde{\mu}_{k+1}, w^{(k+1)}=\hat{w}^{(k+1)}, k=k+1,$
and return to step 1.
\ENDIF
\end{algorithmic}
\end{algorithm}
 The algorithm starts with Step 0 which can be easily initialized by choosing arbitrary $w^{(0)}\in \mathbb{E}\times \mathbb{E}$, $\mu_0>0$ and $\beta>
\max \{\sqrt{\vartheta},\Psi_{\mu_0}(w^{(0)})/\mu_0\}$.  Centering steps 1--3 are crucial to obtain the global convergence rate, while approximate Newton steps 4--5 are necessary to obtain the local convergence rate. As proved in Theorem \ref{theorem:sp_linear_conver}, when $k$ is sufficiently large, Algorithm \ref{alg:IeNICM} updates $w^{(k+1)}=\hat{w}^{(k+1)}$ eventually. Newton equations \eqref{eq:Newton_direction} and \eqref{eq:pure_Newton_direction} of centering steps and approximate Newton steps respectively are solved inexactly. Parameters $\theta_1^{(k)}$ and $\theta_2^{(k)}$ are to control the level of accuracy in solving the Newton equations. As mentioned in the introduction, we remark that finding the inexact solutions satisfying \eqref{eq:Newton_direction} and \eqref{eq:pure_Newton_direction} does not require computing the explicit form of the Hessian matrix $\derD H_{\mu_k}$ and its inverse; instead we can use some Hessian-free Newton type methods or apply some iteration solvers to find the inexact Newton directions.  
\subsection{Global linear convergence}
We first list assumptions that will be used in sequel.  
\begin{assumption}\label{assump:uniform_bounded}
The derivative $\derD H_{\mu_k}(w^{(k)})$ is nonsingular and there exists a constant $C$ such that $\norm{\derD H_{\mu_k}(w^{(k)})^{-1}} \leq  C$ for all $k$.
\end{assumption}
\begin{assumption} \label{assump:F}
We have 
$$\|F(y)-F(x) - \derD F(x)[y-x] \| =o (\|y-x\|) \;\text{for all}\; x,y \in \bbE.$$
\end{assumption}
\begin{assumption} \label{assump:F2}
There exist constant  $\xi>1$ and $L>0$ such that
$$\|F(y)-F(x) - \derD F(x)[y-x] \| \leq L\|y-x\|^{\xi} \;\text{for all}\; x,y \in \bbE$$
\end{assumption}

As mentioned in introduction section, we will prove in Section \ref{subsec:sufficient conds} that if $F$ is monotone then $\derD H_{\mu_k}(w^{(k)})$ is nonsingular. Furthermore, we will extend the sufficient condition obtained in \cite{Chen_Tseng} for the uniform boundedness of $\norm{\derD H_{\mu_k}(w^{(k)})^{-1}}$ to general convex set, and refer the readers to \cite{Bintong,Chua_Hien} for more relaxed conditions that guarantee the boundedness when $X$ is a non-negative orthant, an epigraph of operator norm or an epigraph of nuclear norm.  Assumption \ref{assump:F2} is typical in global convergence analysis of non-interior continuation methods, see e.g., \cite[Proposition 1]{Chen_Tseng}, \cite[Assumption 2]{Bintong}. Assumption \ref{assump:F} will be used to prove the local superlinear convergence, and Assumption \ref{assump:F2} will be used to prove the local $\xi$-order convergence of our proposed method. Assumption \ref{assump:F} is satisfied if Assumption \ref{assump:F2} is satisfied. 
 
Proposition \ref{prop:Newton_direction_bounded} provides bounds for Newton directions; see its proof in \ref{proof1}.
\begin{proposition} \label{prop:Newton_direction_bounded}
Let $w^{(k)}$ be the $k$-th iterate of the Algorithm 1, $\triangle \tilde{w}^{(k)}$ be the solution of \eqref{eq:Newton_direction} and $\triangle \hat{w}^{(k)}$ be the solution of \eqref{eq:pure_Newton_direction}. If  Assumption \ref{assump:uniform_bounded} holds true, then
\begin{itemize}
\item[(i)] $\norm{\triangle \tilde{w}^{(k)}} \leq C(1+\theta_1) \Psi_{\mu_k}(w^{(k)})\leq C(1+\theta_1)\beta \mu_k$, where $\theta_1 = \sup \theta_1^{(k)}$,
\item[(ii)]  $\norm{\triangle \hat{w}^{(k)}} \leq C(1+\theta_2)(\beta + \sqrt{\vartheta}) \mu_k$, where $\theta_2 = \sup \theta_2^{(k)}$.
\end{itemize}
\end{proposition}
We now apply Proposition \ref{prop:bsa_property} together with Proposition \ref{prop:Newton_direction_bounded} to bound the value of $H_\mu(x+\lambda \triangle \tilde{x},y+ \lambda \triangle \tilde{y})$.
\begin{proposition}
\label{prop:F_Phi_property}
Let $ 0\leq \lambda \leq 1$, $r_1=\begin{pmatrix} r_{1x}\\r_{1y} \end{pmatrix}$, $(\triangle \tilde{x}, \triangle \tilde{y})$ be the solution of \eqref{eq:Newton_direction}.
\begin{itemize}
\item[(i)] If the barrier $f(x)$ is $\vartheta$-self-concordant, then
 $$\| \phi_\mu(x+\lambda \triangle \tilde{x}, y+ \lambda \triangle \tilde{y})\| \leq (1-\lambda) \| \phi_\mu(x,y)\| + \cfrac{1}{4\mu} \lambda^2 \norm{ (\triangle \tilde{x},\triangle \tilde{y})}^2 + \lambda\norm{r_{1x}}.
 $$
\item[(ii)] If Assumption \ref{assump:uniform_bounded} holds and the map $F$ satisfies Assumption \ref{assump:F} then
$$\|F(x+\lambda \triangle \tilde{x})-( y+ \lambda \triangle \tilde{y})\| \leq (1-\lambda)\| F(x)-y\| + o(\lambda \| (\triangle \tilde{x},\triangle \tilde{y})\|) + \lambda\norm{r_{1y}}.
$$
\end{itemize}
\end{proposition}
\begin{proof}
(i) From Equation \eqref{eq:Newton_direction}, we get
$$\begin{array}{ll}
x-p_\mu(x-y)&=-(\mathbf{I} - \derD p_\mu(x-y))[\triangle \tilde{x}] - \derD p_\mu(x-y)[\triangle \tilde{y}] + r_{1x}\\
&=-\bracket{\triangle \tilde{x} - \derD p_\mu(x-y)[\triangle \tilde{x}-\triangle \tilde{y}] - r_{1x}}.
\end{array}$$
 Therefore,
\begin{align}
\label{temp1}
\begin{split}
 &\phi_\mu(x+\lambda \triangle \tilde{x}, y+ \lambda \triangle \tilde{y}) - (1-\lambda) \phi_\mu(x,y)\\
  & = x + \lambda \triangle \tilde{x} -p_\mu\big( x-y +\lambda(\triangle \tilde{x}-\triangle \tilde{y})\big) - (x-p_\mu(x-y))- \\ 
  & \qquad \qquad \lambda\big(\triangle \tilde{x} - \derD p_\mu(x-y)[\triangle \tilde{x}-\triangle \tilde{y}]- r_{1x}\big)\\
   & =-p_\mu\big(x-y +\lambda(\triangle \tilde{x}-\triangle \tilde{y})\big) + p_\mu(x-y) + \lambda \derD p_\mu(x-y)[\triangle \tilde{x}-\triangle \tilde{y}]  + \lambda r_{1x}.
\end{split}
\end{align}
Denote $\tilde{\theta}=p_\mu\big(x-y +\lambda(\triangle \tilde{x}-\triangle \tilde{y})\big) - p_\mu(x-y) - \lambda \derD p_\mu(x-y)[\triangle \tilde{x}-\triangle \tilde{y}]$. Using Lagrange's remainder for first order Taylor polynomial,  for all $v\in \mathbb{E}$ we have $t\in (0,1)$ such that 

\begin{align*}
\iprod{v}{\tilde{\theta}}
&=\iprod{v}{\frac12 \derD^2 p_\mu\bracket{x-y+t\lambda \bracket{\triangle \tilde{x}-\triangle \tilde{y}}}\left[\lambda (\triangle \tilde{x}-\triangle \tilde{y}), \lambda (\triangle \tilde{x}-\triangle \tilde{y}) \right]}\\
&\leq \frac{1}{8\mu} \lambda^2 \norm{v} \norm{\triangle \tilde{x}-\triangle \tilde{y}}^2,
\end{align*}
where we have applied Proposition \ref{prop:bsa_property} for the last inequality. Hence
\[\norm{\tilde{\theta}}^2=\iprod{\tilde{\theta}}{\tilde{\theta}}\leq \frac{1}{8\mu} \lambda^2 \norm{\tilde{\theta}} \norm{\triangle \tilde{x}-\triangle \tilde{y}}^2,
\]
which implies $\norm{\tilde{\theta}}\leq \frac{1}{8\mu} \lambda^2 \norm{\triangle \tilde{x}-\triangle \tilde{y}}^2\leq \frac{1}{4\mu} \lambda^2 \norm{(\triangle \tilde{x},\triangle \tilde{y})}^2$. 
Together with \eqref{temp1}, we get (i). 

(ii)  From Equation \eqref{eq:Newton_direction}, we have
$F(x)-y= - \derD F(x)[\triangle \tilde{x}] + \triangle \tilde{y} + r_{1y}.$
This equality together with Assumption \ref{assump:F} yields
\begin{align*}
\begin{split}
 &F(x+\lambda \triangle \tilde{x})- ( y+ \lambda \triangle \tilde{y}) - (1-\lambda) (F(x)-y)\\
 &\qquad=F(x+\lambda \triangle \tilde{x}) - ( y+ \lambda \triangle \tilde{y}) - (F(x)-y) +\lambda(- \derD F(x)[\triangle \tilde{x}] + \triangle \tilde{y} + r_{1y} )\\
 &\qquad = F(x+\lambda \triangle \tilde{x}) -  \derD F(x)[\lambda \triangle \tilde{x}] - F(x) + \lambda r_{1y}\\
 &\qquad = o(\lambda \norm{\triangle \tilde{x}} ) + \lambda r_{1y}\\
 &\qquad = o(\lambda  \| (\triangle \tilde{x},\triangle \tilde{y})\|) + \lambda r_{1y}.
\end{split}
\end{align*} \QEDB
\end{proof}
Finally, we establish the positive lower bounds for $\lambda_k$ and $\gamma_k$ that are necessary to prove the global convergence in Theorem \ref{global}.
\begin{proposition} 
\label{prop:lamda_gamma_property} Let $0\leq \lambda\leq 1$. We consider $\lambda_k$ and $\gamma_k$ in Step 2 and Step 3 of Algorithm \ref{alg:IeNICM} respectively. If Assumptions \ref{assump:uniform_bounded} and \ref{assump:F} are satisfied, then
\begin{itemize}
\item[(i)] There exists $\bar{\lambda}$ such that $\lambda_k \geq \alpha_1\bar{\lambda}$, and 
\item[(ii)] $\gamma_k \geq \alpha_2 \bar{\gamma}$ where $\bar{\gamma}=\min\left\{1, \cfrac{\beta \sigma\bar{\lambda}}{\beta+\sqrt{\vartheta}}\right\}.$
\end{itemize}
\end{proposition}

\begin{proof}
(i)
By Proposition \ref{prop:F_Phi_property}, we get
\begin{align*}
\Psi_{\mu_k}(w^{(k)} + \lambda \triangle \tilde{w}^{(k)})\leq (1-\lambda)\Psi_{\mu_k}(w^{(k)})+ o( \lambda \| \triangle \tilde{w}^{(k)}\|) + \frac{1}{4 \mu_k} \; \lambda^2  \| \triangle \tilde{w}^{(k)}\|^2+\lambda\sqrt2 \norm{r_1^{(k)}}.
\end{align*}
Using Proposition \ref{prop:Newton_direction_bounded}(i) and remind that $\norm{r_1^{(k)}}\leq \theta_1^{(k)} \norm{H_{\mu_k}\bracket{w^{(k)}}}$, we deduce that there exists a function $\varpi(\cdot)  $ such that $\lim_{x\to 0}\frac{\varpi(x)}{x}=0$ and 
\begin{align} 
\label{temp02}
\begin{split}
&\Psi_{\mu_k}(w^{(k)} + \lambda \triangle \tilde{w}^{(k)})\\
&\leq (1-\lambda)  \Psi_{\mu_k}(w^{(k)})+ \varpi( \lambda\Psi_{\mu_k}(w^{(k)})) +  \frac14 \lambda^2 C^2(1+\theta_1)^2\beta \Psi_{\mu_k}(w^{(k)})+ \lambda\sqrt2 \theta_1\Psi_{\mu_k}(w^{(k)}).
\end{split}
\end{align}
Note that $\Psi_{\mu_k}(w^{(k)}) \leq \beta\mu_k \leq \beta \mu_0$, i.e., $\Psi_{\mu_k}(w^{(k)})$ is bounded by $\beta \mu_0$. Hence, there exists $\bar{\lambda}\geq 0$ such that for all $0\leq \lambda\leq \bar{\lambda}$, we have 
$$\cfrac{\varpi( \lambda\Psi_{\mu_k}(w^{(k)}))}{\lambda\Psi_{\mu_k}(w^{(k)})} + \frac14 C^2 (1+\theta_1)^2 \beta \lambda\leq 1-\sigma -\sqrt2 \theta_1, 
$$
which together with \eqref{temp02} leads to $\Psi_{\mu_k}(w^{(k)} + \lambda \triangle \tilde{w}^{(k)}) \leq (1-\sigma\lambda)  \Psi_{\mu_k}(w^{(k)})$. Therefore, for all $ 0\leq \lambda \leq \bar{\lambda}$ the line search criteria \eqref{ieq1:line_search} for centering step  holds true.  Then we get (i). 

(ii) We note that
\begin{align*}
&\Psi_{(1-\gamma)\mu_k}(\tilde{w}^{(k+1)}) - \Psi_{\mu_k}(\tilde{w}^{(k+1)}) \\
&\qquad=  \norm{\phi_{(1-\gamma)\mu_k}(\tilde{w}^{(k+1)})} - \norm{\phi_{\mu_k}(\tilde{w}^{(k+1)})} \\
 &\qquad\leq \norm{ \phi_{(1-\gamma)\mu_k}(\tilde{w}^{(k+1)})- \phi_{\mu_k}(\tilde{w}^{(k+1)})}\\
 &\qquad=\| p_{(1-\gamma)\mu_k}(\tilde{x}^{(k+1)}-\tilde{y}^{(k+1)}) - p_{\mu_k}(\tilde{x}^{(k+1)}-\tilde{y}^{(k+1)})\|.
\end{align*}
Using the Lipschitz continuity of $p_\mu$ (see Theorem \ref{theorem:bsa}), we have
\begin{align}
\label{temp_02}
\begin{split}
\Psi_{(1-\gamma)\mu_k}(\tilde{w}^{(k+1)})\leq \Psi_{\mu_k}(\tilde{w}^{(k+1)}) +  \sqrt{\vartheta} \gamma\mu_k . 
\end{split}
\end{align}
On the other hand, as proved above, we have $\Psi_{\mu_k}(\tilde{w}^{(k+1)}) \leq (1-\sigma \bar{\lambda}) \Psi_{\mu_k}(w^{(k)})$. Hence, using the fact $\Psi_{\mu_k}(w^{(k)}) \leq \beta \mu_k $, we derive from \eqref{temp_02} that 
 $$\Psi_{(1-\gamma)\mu_k}(\tilde{w}^{(k+1)})\leq (1-\sigma \bar{\lambda}) \Psi_{\mu_k}(w^{(k)}) + \sqrt{\vartheta}\gamma \mu_k \leq \beta (1-\sigma \bar{\lambda} + \cfrac{\sqrt{\vartheta}\gamma}{\beta})\mu_k.$$
Finally, we get $\Psi_{(1-\gamma)\mu_k}(\tilde{w}^{(k+1)}) \leq \beta(1-\gamma)\mu_k$, i. e., $ \tilde{w}^{(k+1)} \in \mathcal{N}(\beta,(1-\gamma)\mu_k)$ for all $0\leq \gamma \leq \bar{\gamma}$. The result (ii) follows then. \QEDB
\end{proof}
\begin{remark}
If we assume Assumption \ref{assump:F2} holds, then we get the following inequality, which is stronger than the inequality  in Proposition \ref{prop:F_Phi_property}(ii),
\begin{align*}
\|F(x+\lambda \triangle \tilde{x})-( y+ \lambda \triangle \tilde{y})\| &\leq (1-\lambda)\norm{F(x)-y}+ L\lambda^\xi\norm{\triangle \tilde{x}}^\xi + \lambda\norm{r_{1y}}\\
&\leq (1-\lambda)\norm{F(x)-y}+L\lambda^\xi\norm{(\triangle \tilde{x},\triangle \tilde{y})}^\xi + \lambda\norm{r_{1y}}.
\end{align*}
Consequently, we obtain the following inequality which is stronger than \eqref{temp02}
\begin{align*}
&\Psi_{\mu_k}(w^{(k)} + \lambda \triangle \tilde{w}^{(k)})\\
&\,\leq 	(1-\lambda)\Psi_{\mu_k}(w^{(k)}) + \big(L \lambda^\xi(C(1+\theta_1)\beta\mu_k)^{\xi-1} \\
&\qquad + \frac14 \lambda^2 C(1+\theta_1) \beta\big) C(1+\theta_1) \Psi_{\mu_k}(w^{(k)}) + \lambda\sqrt2\theta_1\Psi_{\mu_k}(w^{(k)}) \\
&\,\leq (1-\lambda)\Psi_{\mu_k}(w^{(k)})\\
&\quad+ \lambda \left\{ \lambda^a(L(C(1+\theta_1)\beta\mu_0)^{\xi-1} + \frac14C(1+\theta_1)\beta)  C(1+\theta_1) + \sqrt2 \theta_1\right\}  \Psi_{\mu_k}(w^{(k)}),
\end{align*}
where $a=\min\{ \xi-1,1\}$. Hence the value of $\bar{\lambda}$ in Proposition \ref{prop:lamda_gamma_property} (i) is chosen to be 
$$\bar{\lambda}=\min\left\{1,\sqrt[a]{\frac{1-\sigma-\sqrt2 \theta_1}{C(1+\theta_1) \left(L(C(1+\theta_1)\beta\mu_0)^{\xi-1} + \frac14C(1+\theta_1)\beta\right)}} \right\}.$$ 
\end{remark}
The following theorem states the global linear convergence of our algorithm. Its proof is  similar to that of \cite[Theorem 1]{Bintong}; we hence omit the details here.
\begin{theorem} 
\label{global}
Assuming Assumption \ref{assump:uniform_bounded}
and Assumption \ref{assump:F} are satisfied, then we have
\begin{itemize}
\item[(i)] For all $k\geq 0$,   $\mu_{k+1}\leq (1 -\alpha_2\bar{\gamma}) \mu_k$ and $\mu_k \leq \mu_0(1-\alpha_2\bar{\gamma})^k$, where $\bar\gamma$ is defined in Proposition \ref{prop:lamda_gamma_property}.
\item[(ii)] The sequence $\left\{H_0(w^{(k)})\right\}_{k\geq 0}$ converges to 0 $R$-linearly.
\item[(iii)] The sequence $\left\{w^{(k)}\right\}_{k\geq 0}$ is a Cauchy sequence converging to a solution $w^*=(x^*,y^*)$ of the $VI$ \eqref{eq:compl_equation}.
\end{itemize}
\end{theorem}

\subsection{Local superlinear/$\xi$-order convergence}
\label{subsec:localrate}
We denote $z^{(k)}=x^{(k)} - y^{(k)}$  and $z^*=x^*-y^*$, where $(x^*,y^*)$ is any limit point of the sequence $\left\{ \bracket{x^{(k)},y^{(k)}}\right\}_{k\geq 0}$. We now use Theorem \ref{theorem:seq_Jacob_conver} to establish the local convergence of Algorithm \ref{alg:IeNICM}. 
We will see in Section \ref{sec:application} that the condition $\lim_{(z,\mu)\to (z^*,0)} \derD p_\mu(z)=T^* $ is satisfied in many specific convex sets. 
\begin{theorem}
\label{theorem:sp_linear_conver}
Suppose Assumption \ref{assump:uniform_bounded} and Assumption \ref{assump:F} hold. If the derivative  $\derD p_\mu(z)$ converges to a linear operator $T^*$ when $(z,\mu)$ goes to $(z^*,0)$, Algorithm \ref{alg:IeNICM} generates infinite sequence $\left\{w^{(k)}\right\}_{k\geq 0}$ and $r_2^{(k)} = o\lrpar{\norm{H_0(w^{(k)})}}$, then the sequence $\left\{w^{(k)}\right\}_{k\geq 0}$ converges superlinearly to $(x^*,y^*)$.  

Moreover, if $\norm{\derD p_\mu(z) -T^*} = O\bracket{\norm{(z-z^*,\mu)}^{\xi-1}}$ and $r_2^{(k)} = O(\norm{H_0(w^{(k)})} ^{\xi})$ with $\xi >1$ as given in Assumption \ref{assump:F2}, then the convergence is of $\xi$-order.
\end{theorem}
\begin{proof}
By Theorem \ref{theorem:seq_Jacob_conver}, we can deduce  that $T^*=\derD \Pi_X(z^*)$.

First, we prove that $\norm{ \hat{w}^{(k+1)} - w^*} = o\bracket{\norm{w^{(k)}-w^*}}$. By Equation \eqref{eq:pure_Newton_direction}, 
\begin{align*}
\begin{split}
&\qquad\norm{w^{(k)} + \triangle \hat{w}^{(k)} - w^*}\\
 &\qquad=  \norm{w^{(k)}- \derD H_{\mu_k}(w^{(k)})^{-1} (H_0(w^{(k)})-r_2^{(k)} )-w^* }\\
&\qquad \leq \norm{ \derD H_{\mu_k}(w^{(k)})^{-1} }  \norm{ \derD H_{\mu_k}(w^{(k)})[w^{(k)}-w^*] -H_0(w^{(k)}) + r_2^{(k)}}\\
& \qquad\leq C\bracket{q_1 + q_2+\norm{r_2^{(k)}}},
\end{split}
\end{align*}
where
\begin{align}
\label{r_1}
\begin{split}
q_1 & =\norm{\derD F(x^{(k)})[x^{(k)}-x^*] - [y^{(k)}-y^*] - (F(x^{(k)})-y^{(k)})}\\
 &= \norm{ F(x^{(k)}) - \derD F(x^{(k)})[x^{(k)}-x^*] - F(x^*)} \\
 &= o\bracket{\norm{ x^{(k)}-x^*}} \quad \{\text{\textit{ by Assumption}}\; \ref{assump:F}  \} \\
 &= o\bracket{\norm{w^{(k)}-w^*}},
\end{split}
\end{align}
and
\begin{align}
\label{r_2}
\begin{split}
q_2 &= \norm{\big( \mathbf{I} - \derD p_{\mu_k}(z^{(k)})\big)[x^{(k)}-x^*] + \derD p_{\mu_k}(z^{(k)})[ y^{(k)}-y^*] - \big( x^{(k)}- \Pi_K(z^{(k)})\big)}\\
&= \norm{\Pi_K( z^{(k)}) - \Pi_K(z^*)- \derD p_{\mu_k}(z^{(k)})[x^{(k)} - x^* - (y^{(k)}-y^*)]} \\
& \leq \norm{\Pi_K( z^{(k)})- \derD \Pi_K(z^*)[z^{(k)} - z^*]-\Pi_K(z^*) }\\
&\qquad+ \norm{ \big(\derD \Pi_K(z^*) -  \derD p_{\mu_k}(z^{(k)})\big) [z^{(k)} - z^*]}.
\end{split}
\end{align}
By Theorem \ref{theorem:seq_Jacob_conver}, we get $q_2 =   o\bracket{\norm{z^{(k)} - z^* } } =o\bracket{\norm{w^{(k)}-w^*}}$. Together with $r_2^{(k)} = o\bracket{\norm{H_0(w^{(k)})}}$, we deduce that
$ \norm{ w^{(k)}+ \triangle \hat{w}- w^*} =o \bracket{\norm{w^{(k)}-w^*}}.$ This implies
\begin{align} \label{eq:smallO}
\norm{ w^{(k)}+ \triangle \hat{w}- w^*} = \tau_k  \norm{ w^{(k)} -  w^*},
\end{align}
where $ \tau_k$ is a sequence converging to 0. Furthermore,
$$
\norm{w^{(k)}-w^*}=\norm{ w^{(k)}+\triangle \hat{w} -w^* - \triangle \hat{w}} \leq \norm{ w^{(k)}+\triangle \hat{w}-w^*} +\norm{\triangle \hat{w}}.
$$
Hence, $\norm{ w^{(k)}-w^*}\leq  \tau_k\norm{w^{(k)}-w^*}+  \norm{\triangle \hat{w}}$. 
Applying Proposition \ref{prop:Newton_direction_bounded} (ii), we have
\begin{align*}
\norm{ w^{(k)}-w^*}& \leq  \tau_k \norm{ w^{(k)}-w^*}+  C(1+\theta_2)(\beta + \sqrt{\vartheta}) \mu_k \\
&\leq \cfrac12 \norm{ w^{(k)}-w^*}+  C(1+\theta_2)(\beta + \sqrt{\vartheta}) \mu_k 
\end{align*}
 for sufficiently large $k$. Hence for sufficiently large $k$ the following inequality is satisfied
\begin{align} \label{temp4}
 \| w^{(k)}-w^*\| \leq 2  C(1+\theta_2)(\beta + \sqrt{\vartheta}) \mu_k.
\end{align}
Now we prove that  $\hat{w}^{(k+1)} \in \mathcal{N}(\beta,(1-\gamma_k)\mu_k)$ for sufficiently large $k$, then $w^{(k)} = \hat{w}^{(k)}$ eventually. Using the property $\sqrt2\norm{(a,b)} \geq \norm{a} + \norm{b}$, similarly to \eqref{temp3} we can  prove
\begin{align} \label{temp50}
\Psi_{(1-\gamma_k)\mu_k}\bracket{\hat{w}^{(k+1)}} \leq  \sqrt 2 \norm{H_0(\hat{w}^{(k+1)})} + (1-\gamma_k)\mu_k \sqrt{\vartheta}.
\end{align}
Using the Lipschitz continuity of $H_0(w)$ near $w^*$ with Lipschitz constant $L_1$,  Expression \eqref{eq:smallO} and Inequality \eqref{temp4} we then get
\begin{align*}
\begin{split}
&\Psi_{(1-\gamma_k)\mu_k}(\hat{w}^{(k+1)})\\
& \leq \sqrt2 L_1 \norm{\hat{w}^{(k+1)} - w^*} +  (1-\gamma_k)\mu_k \sqrt{\vartheta}\\
&= \sqrt2 L_1  \tau_k  \norm{w^{(k)} -  w^*} +  (1-\gamma_k)\mu_k \sqrt{\vartheta}\\
 & \leq \sqrt2 L_1 \tau_k 2 C(1+\theta_2) (\beta + \sqrt{\vartheta})  \mu_k +   (1-\gamma_k)\mu_k \beta - (1-\gamma_k)\mu_k(\beta -\sqrt{\vartheta})\\
 &\leq  (1-\gamma_k)\mu_k \beta + \mu_k ( 2\sqrt2 L_1 \tau_k C(1+\theta_2)(\beta + \sqrt{\vartheta})- (1-\alpha_2)(\beta-\sqrt{\vartheta}) )
\end{split}
\end{align*}
where we have used the fact $\gamma_k \leq \alpha_2$ from Step 3 of Algorithm \ref{alg:IeNICM} (if $\gamma_k=1$ then the algorithm terminates finitely). Finally, for sufficiently large $k$  the value $2\sqrt2 L_1 \tau_k C(1+\theta_2)(\beta + \sqrt{\vartheta})- (1-\alpha_2)(\beta-\sqrt{\vartheta})$ is negative number since $\tau_k \rightarrow 0$ and $\beta>\sqrt{\vartheta}$ was chosen; we hence get $$ \Psi_{(1-\gamma_k)\mu_k}(\hat{w}^{(k+1)})\leq \beta  (1-\gamma_k)\mu_k.$$

For locally $\xi$-order convergence, we use the result of \cite[Lemma 7]{Bintong} which states that if the algorithm generates the $k$-th iterate by the approximate Newton step then $\mu_k=O(\|w^{(k)}-w^*\|)$. Indeed, based on the approximate Newton step, we have $w^{(k)} \not\in \mathcal{N}(\beta,\alpha_3\mu_k).$ Moreover, similarly to \eqref{temp50}, we get $\Psi_{\alpha_3 \mu_k}(w^{(k)}) \leq \sqrt 2 \norm{H_0({w}^{(k)})} + \alpha_3 \mu_k \sqrt{\vartheta}$. Hence 
$$\beta \alpha_3 \mu_k < \Psi_{\alpha_3 \mu_k}(w^{(k)}) \leq \sqrt 2 \norm{H_0({w}^{(k)})} + \alpha_3 \mu_k\sqrt{\vartheta}.$$
This implies 
$$\mu_k<\cfrac{\sqrt2\norm{H_0({w}^{(k)})}}{\alpha_3 (\beta-\sqrt{\vartheta})}\leq \cfrac{\sqrt2 L_1\norm{{w}^{(k)}-w^*}}{\alpha_3 (\beta-\sqrt{\vartheta})}.
$$
Appealing to this fact, if $\norm{\derD p_\mu(z) -T^*} = O(\|(z-z^*,\mu)\|^{\xi-1})$ then from Theorem \ref{theorem:seq_Jacob_conver} and Expression \eqref{r_2} we yield $q_2= O(\|w^{(k)}-w^*\|^\xi)$. From Expression \eqref{r_1} and Assumption \ref{assump:F2}, we have $q_1=O(\norm{w^{(k)}-w^*}^\xi)$. All together with $\norm{r_2^{(k)}} =  O(\|w^{(k)}-w^*\|^\xi)$, we deduce 
$ \|\hat{w}^{(k+1)} - w^*\| = O(\|w^{(k)}-w^*\|^\xi).$ \QEDB
\end{proof}
\subsection{Sufficient conditions for Assumption \ref{assump:uniform_bounded}}
Assumption \ref{assump:uniform_bounded} is crucial in obtaining the global linear convergence of our algorithm. The following proposition provides some sufficient conditions that make Assumption \ref{assump:uniform_bounded} satisfied for general closed convex sets. 
\label{subsec:sufficient conds}
\begin{proposition}  \label{prop:strmonotone}
(i) If $F$ is monotone then $\derD H_\mu(x,y)$ is nonsingular for all $\mu>0$ and $ (x,y)\in \mathbb{E}\times \mathbb{E}$. 

(ii) If $F$ is strongly monotone and assume that $\left\{\derD F\lrpar{x^{(k)}}\right\}_{k\geq 0}$ is  bounded, 
then  Assumption \ref{assump:uniform_bounded} holds.
\end{proposition}
\begin{proof}
(i)  We remind that  $\der p_\mu(z)=[I+\mu^2  \nabla^2 f(p(z))]^{-1}$ (see Formula \eqref{eq:derofp}). We also note that $\derD p_\mu(z)u =\der p_\mu(z) u, \forall u\in \mathbb{E}$. Then it is easy to see that $\mathbf{0} \prec \derD p_\mu(z) \prec \mathbf{I}$ for all $z\in \mathbb{E}$. 
Denote $\derD p_\mu(x-y) = \bar{\derD}$. We have  
$\derD H_{\mu}(x,y)=\begin{pmatrix}
\mathbf{I}-\bar{\derD}  & \bar{\derD}  \\
\derD F(x) & -\mathbf{I}
\end{pmatrix}.$
The system of linear equations  $\derD H_{\mu}(x,y)[(u,v)]=0$ is equivalent to
$\left\{ \begin{array}{ll}
(\mathbf{I}-\bar{\derD} )[u] + \bar{\derD}  [v]=0\\
\derD F(x)  [u]  - v=0,
\end{array}\right.
$
which can be rewritten as 
$\left\{ \begin{array}{ll}
(\mathbf{I}-\bar{\derD}  + \bar{\derD} (\derD F(x)))[u]=0\\
v=\derD F(x) [u].
\end{array}\right.
$
It follows from $\mathbf{0} \prec \bar{\derD}  \prec \mathbf{I} $  and monotonicity of $F$ that $ (\mathbf{I}-\bar{\derD} )\bar{\derD} \succ \mathbf{0}$ and $\bar{\derD}  (\derD F(x) ) \bar{\derD} \succeq \mathbf{0}$. 
Therefore,
\begin{align*}
(\mathbf{I}-\bar{\derD}  + \bar{\derD} (\derD F(x)))[u]=0 &\Leftrightarrow (\mathbf{I}-\bar{\derD}  + \bar{\derD}  (\derD F(x))\bar{\derD}  \bar{\derD} ^{-1}[u]=0\\
&\Leftrightarrow  \big( (\mathbf{I}-\bar{\derD} )\bar{\derD}  + \bar{\derD}  (\derD F(x)) \bar{\derD}  \big)\bar{\derD} ^{-1} [u]=0\\
&\Leftrightarrow  \bar{\derD}^{-1}[u]=0 \\
&\Leftrightarrow  u=0.
\end{align*}
We deduce that $\derD H_{\mu}(x,y)[(u,v)]=0$ has the unique solution $(u,v)=(0,0)$. The result (i) follows then.

(ii) The following proof is inspired by the proof of \cite[Proposition 4.4]{Chua_Yi}. 

Denote $M=\derD F(x^{(k)})$, $\bar{D}^{(k)}=\derD p_\mu(x^{(k)}-y^{(k)})$ and $\bar{D}^{(k)}[u]=\tilde{u}$. Since $F$ is strongly monotone, then there exists a constant $\varrho$ such that $\iprod{Md}{d} =\iprod{M^Td}{d} \geq \varrho \norm{d}^2$ for all $d\in \mathbb{E}$. Furthermore, we are considering the barrier $f$ with positive semidefinite  $ \nabla^2 f$, hence  $\iprod{M^Td + \mu^2 \nabla^2 f\bracket{x^{(k)}-y^{(k)}}d }{d} \geq \varrho \norm{d}^2$, $\forall d\in \mathbb{E}$. By Cauchy-Schwarz inequality, 
$$\iprod{M^Td + \mu^2 \nabla^2 f\bracket{x^{(k)}-y^{(k)}}d}{d} \leq \norm{M^Td + \mu^2 \nabla^2 f\bracket{x^{(k)}-y^{(k)}}d }\norm{d},$$ we then deduce $\norm{M^Td + \mu^2 \nabla^2 f\bracket{x^{(k)}-y^{(k)}}d} \geq \varrho \norm{d}$ for all $d\in \mathbb{E}$. 

Let $m_F$ be the constant such that $\norm{\derD F(x^{(k)})}\leq m_F$ for all $k\geq 0$. 
For arbitrary $u\in \mathbb{E}$, we consider 2 cases

Case 1: $\norm{\tilde{u}} \geq \cfrac{1}{1+m_F+\varrho} \norm{u}$. We note that $$\bar{D}^{(k)}[u]=\tilde{u} \Leftrightarrow \bracket{I+\mu^2 \nabla^2 f\bracket{x^{(k)}-y^{(k)}}}[\tilde{u}] =u.$$
Then we have
\begin{align*} 
\norm{\bracket{\mathbf{I}-\bar{\derD}^{(k)}  + M^T \bar{\derD}^{(k)}}[u]}&=\norm{\tilde{u}+\mu^2 \nabla^2 f(x^{(k)}-y^{(k)}) \tilde{u}-\tilde{u} + M^T \tilde{u}}\\
&\geq \varrho \norm{\tilde{u}} \geq \cfrac{\varrho}{1+m_F + \varrho}  \norm{u}.
\end{align*}

Case 2: $\norm{\tilde{u}} < \cfrac{1}{1+m_F+\varrho} \norm{u}$. We have
\begin{align*} 
\norm{u -\bar{\derD}^{(k)}[u] +  M^T \bar{\derD}^{(k)} [u] }&\geq \norm{u}-\norm{\tilde{u}} - \norm{M^T \tilde{u}}\\
&\geq \norm{u} - (1+m_F) \norm{\tilde{u}}\\
&\geq \cfrac{\varrho}{1+m_F + \rho}\norm{u}.
\end{align*}
We have proved that $\norm{u -\bar{\derD}^{(k)} u +  M^T \bar{\derD}^{(k)} u } \geq \cfrac{\varrho}{1+m_F + \varrho}\norm{u} $ for all $u\in \mathbb{E}$. On the other hand, we note that the smallest singular value of arbitrary linear operator $L$, which equals to $\min\limits_{u\ne 0} \left\{ \cfrac{\norm{L u}}{\norm{u}}\right\}$, is invariant under taking transpose. Hence, for all $u\in \mathbb{E}$ we have 
\begin{align}\label{eq:unsi}
 \norm{u -\bar{\derD}^{(k)} u +  \bar{\derD}^{(k)} Mu} \geq \cfrac{\varrho}{1+m_F + \varrho}\norm{u}.
\end{align}  
For fixed $(r,s)\in \mathbb{E}\times \mathbb{E}$, let $\derD H_\mu(x^{(k)},y^{(k)})^{-1}[(r,s)]=(u,v)$.  It follows from $$\derD H_\mu(x^{(k)},y^{(k)})[(u,v)]=(r,s)$$ that 
$(\mathbf I-\bar{\derD}^{(k)} )[u]+\bar{\derD}^{(k)} v=r, \quad \derD F(x^{(k)})[u]-v=s.$
Therefore, 
\begin{align} \label{eq:uns}
\left\{ \begin{array}{ll}
\bracket{\mathbf{I}-\bar{\derD}^{(k)}  + \bar{\derD}^{(k)}  (\derD F(x^{(k)}))}[u]=r + \bar{\derD}^{(k)} [s],\\
v=\derD F(x^{(k)}) [u] -s
\end{array}\right.
\end{align}
From the first equation of \eqref{eq:uns} and Inequality \eqref{eq:unsi}, we deduce 
$$\norm{u}\leq \cfrac{1+m_F+\varrho}{\varrho} \norm{r+ \bar{\derD}^{(k)} s}=O(\norm{(r,s)}).
$$
And the second equation of \eqref{eq:uns} implies that 
$$\norm{v}=\norm{\derD F(x^{(k)})u-s} \leq m_F\norm{u} + \norm{s} = O(\norm{(r,s}).
$$
Consequently, $\norm{\derD H_{\mu}(x^{(k)},y^{(k)})^{-1}}$ is uniformly bounded. \QEDB
\end{proof}

\section{Application to specific convex sets} 
\label{sec:application}
In this section, we use notation $(x^*,y^*)$ and $z^*$ as in Section \ref{subsec:localrate}. We now apply our  result in Theorem \ref{theorem:sp_linear_conver} to specific convex sets. In particular, we choose appropriate barrier functions to formulate the corresponding barrier-based smoothing approximations $p_\mu(\cdot)$ of the projection onto the specific convex sets, and we then verify the condition $\norm{\derD p_\mu(z)-T^*} = O\bracket{\norm{(z-z^*,\mu)}}$ in Theorem \ref{theorem:sp_linear_conver} so that the local quadratic convergence of Algorithm \ref{alg:IeNICM} is assured by Theorem \ref{theorem:sp_linear_conver}.  We recall that $T^*=\derD \Pi_X(z^*)$ as proved in Theorem \ref{theorem:seq_Jacob_conver}. We also prove in this section that  when $X$ is a non-negative orthant, a positive semidefinite cone, an epigraph of matrix operator norm or an epigraph of matrix nuclear norm, then differentiability of the projector $\Pi_X(\cdot)$ at $z^*$ is equivalent to strict complementarity of $(x^*,y^*)$. 

To construct the smoothing approximation, throughout this section, we use the following $\vartheta$-self-concordant barriers $f$ for $X$. 
\begin{enumerate}
\item When $X$ is nonnegative orthant $\mathbb{R}^n_+$, we use $f(x)=-\sum\limits_{i=1}^n\log x_i
$.
\item When $X$ is positive semidefinite cone $\mathbb{S}^n_+$, we use $f(x)=-\logdet x$.
\item When $X$ is polyhedral set $P(A,b)=\{x \in \mathbb{R}^n: A x\geq b \}$ for some matrix $A \in \mathbb{R}^{m\times n}$ and vector $b\in \mathbb{R}^n$, we use $f(x)= - \sum\limits_{i=1}^m \log(A_i x -b_i)$.
\item When $X$ is epigraph of matrix operator norm cone
$$\mathcal K_{m,n}=\{(t,x) \in R_+ \times R^{m\times n}: t \geq \| x\|\},\; \text{with}\; m\leq n,
$$
we use (see \cite[Part 5.4.6]{Yurii})
$f(t,x)=-\logdet\: \begin{pmatrix} & tI_n & x^T\\&x &tI_m\end{pmatrix}.$
\item When $X$ is epigraph of matrix nuclear norm 
$$K_{m,n}^\sharp =\{(t,x)\in \mathbb{R}_+ \times \mathbb{R}^{m\times n}: t \geq \norm{x}_* \} \; \text{with}\; m\leq n,
$$
we use the modified Fenchel barrier function
$$
f^\sharp(s_o, s)= -\inf\{s_o x_o +
\trace(s^T x) + f(x_o, x) : (x_o, x) \in K_{m,n}\}.
$$
\end{enumerate}
We use \cite[Definition 2]{Pataki2001} for the definition of strict complementary solutions of VI. Specifically,  considering the VI \eqref{eq:compl_equation} over a convex cone $K$, a pair of feasible primal-dual solution $(x,y)$ of \eqref{eq:compl_equation} is strictly complementary if $x\in\relint(\mathcal{F}) $ and $y\in \relint (\mathcal{F}^\vartriangle)$ for some face $\mathcal{F}$ of $K$. If $(0,y)$ is feasible for $y \in \inter (K^\sharp)$ or $(x,0)$ is feasible for $x \in\inter (K)$ then the corresponding pair is also called strictly complementary. Here $\mathcal{F}^\vartriangle=\{v\in K^\sharp: \forall u\in \mathcal{F}, \iprod{v}{u}=0\}$ is the complementary face of $\mathcal{F}$. The following theorem states our main result. 
\begin{theorem}\label{theorem:verify}
When $X$ is a non-negative orthant, a positive semidefinite cone, a polyhedral set, an epigraph of matrix operator norm or an epigraph of matrix nuclear norm, then the following statements are equivalent. 
\begin{itemize}
\item[(i)] The projector $\Pi_X(\cdot)$ is differentiable at $z^*$.  
\item[(ii)] The derivative $\derD p_\mu(z)$ converges to $\derD \Pi(z^*)$ when $(z,\mu)$ converges to $(z^*,0)$, and we then have
$$ \norm{\derD p_\mu(z)-\derD \Pi_X(z^*)}=O\bracket{\norm{(z-z^*,\mu)}}.
$$ 
\end{itemize}
When $X$ is a non-negative orthant, a positive semidefinite cone, an epigraph of matrix operator norm or an epigraph of matrix nuclear norm, then the above statements are further equivalent to strict complementarity of $(x^*,y^*)$.
\end{theorem}
Although Theorem \ref{theorem:verify} recovers the local quadratic convergence of non-interior continuation method for non-negative orthant and positive semidefinite cone, it is worth noting that Theorem \ref{theorem:verify} provides a new technique in proving these local convergence rate. We leave the proof of Theorem \ref{theorem:verify} for these cases (non-negative orthant and positive semidefinite cone) to  \ref{proof2}. We now provide the proofs for the remaining cases of $X$. We see that Theorem \ref{theorem:seq_Jacob_conver} already shows that Statement $(i)$ is a consequence of Statement $(ii)$. We now prove the inverse direction and prove the equivalence to strict complementarity of $(x^*,y^*)$ case by case.

\subsection{Polyhedral set $P(A,b)$} \label{sec:polyhedral_set}
Before going to the detailed proof for the case of polyhedral sets, we provide some properties that will be used later.
\subsubsection{Some preliminaries for polyhedral set}
Denote
 $$\mathcal{I}_0 =\big\{\mathcal{I} \subset \{1,\ldots,m \}: \exists x\in \mathbb{R}^n \;\text{such that}\; A_i x = b_i\; \forall i\in \mathcal{I}\;  \text{and} \; A_i x > b_i \;\forall i \not\in \mathcal{I} \big \}.$$ 
Proposition \ref{prop:polyhedral_properties} summarizes some results from \cite[Part 4.1]{Facchinei} and \cite[Lemma 5]{Pang}.
\begin{proposition} \label{prop:polyhedral_properties}
\begin{itemize}
\item[(i)]Each nonempty face $\mathcal{F}_\mathcal{I}$ of $P(A,b)$ defines an index set $\mathcal{I}\in \mathcal{I}_0$, and vice versa :
$$\mathcal{F}_\mathcal{I} = \{x \in P(A,b): A_i x=b_i, \forall i\in \mathcal{I}.\}
$$
\item[(ii)] For each $x\in \relint (\mathcal{F}_\mathcal{I})$, the normal cone of $P(A,b)$ at $x$ is defined by
$$ \mathcal{N}_I=\coni\{- A_i^T, i\in \mathcal{I}\},
$$
which is independent of $x$ and depends only on the face $\mathcal{F}_\mathcal{I}$.
\item[(iii)] It holds that
$$ \bigcup\limits_{\mathcal{I} \in \mathcal{I}_0} \mathcal{F}_\mathcal{I} + \mathcal{N}_\mathcal{I} = \mathbb{R}^n.
$$
Moreover, if $\mathcal{I},\mathcal{J}$ are distinct index sets in $\mathcal{I}_0$ such that  
$$P_\mathcal{IJ} = (\mathcal{F}_\mathcal{I} + \mathcal{N}_\mathcal{I}) \cap (\mathcal{F}_\mathcal{J} + \mathcal{N}_\mathcal{J}) \ne \emptyset,
$$
then
\begin{itemize}
\item[(a)] $P_\mathcal{IJ}=(\mathcal{F}_\mathcal{I} \cap \mathcal{F}_\mathcal{J}) + (\mathcal{N}_\mathcal{I} + \mathcal{N}_\mathcal{J})$; and 
\item[(b)] $P_\mathcal{IJ}$ is a common face of $\mathcal{F}_\mathcal{I} + \mathcal{N}_\mathcal{I}$ and $\mathcal{F}_\mathcal{J} + \mathcal{N}_\mathcal{J}$.
\end{itemize}
\item[(iv)] For each $x\in \mathcal{F}_\mathcal{I} + \mathcal{N}_\mathcal{I}$, we have $\Pi_{P(A,b)}(x)= \Pi_{\mathcal{S}_\mathcal{I}}(x)$, where
$$\mathcal{S}_\mathcal{I} = \affine (\mathcal{F}_\mathcal{I})=\{x\in \mathbb{R}^n : A_i x= b_i, \forall i\in \mathcal{I} \}.
$$
The projector is directionally differentiable everywhere
$$\Pi'_{P(A,b)} (x;d)= \Pi_{\mathcal{C}}(d)
$$
where $\mathcal{C} =\mathcal{C}(x; P(A,b))=\mathcal{T}(\bar{x}; P(A,b)) \cap (\bar{x}-x)^\perp,$ with  $\bar{x}=\Pi_{P(A,b)}(x)$, is the critical cone of $P(A,b)$ at $x$.
And  the projector is Fr{\'e}chet-differentiable at $x$ if and only if $x\in \inter ( \mathcal{F}_\mathcal{I} + \mathcal{N}_\mathcal{I})$.
\end{itemize}
\end{proposition}
Gradient and Hessian of the barrier $f(x)= - \sum\limits_{i=1}^m \log(A_i x -b_i)
$ is
$$\nabla f(x)= -\sum\limits_{i=1}^m \cfrac{1}{A_i x -b_i} A_i^T,\quad \nabla^2 f(x)= \sum\limits_{i=1}^m \cfrac{1}{(A_ix-b_i)^2} A_i^T A_i.
$$
The barrier-based smoothing approximation $p_\mu(z)=x$ is then defined by
\begin{align}\label{eq:smooth_equa_polyhedral}
x - \mu^2 \sum\limits_{i=1}^m \cfrac{1}{A_i x -b_i} A_i^T = z
\end{align}

\begin{proposition} \label{prop:limit_polyhedral}
Let $z^*$ be a differentiable point of the projector $\Pi_{P(A,b)}$ and $\mathcal{F}_{\mathcal{I}^*}$ be its neighbor face, i.e., $z^*\in \mathcal{F}_{\mathcal{I}^*} + \mathcal{N}_{\mathcal{I}^*}$. Let $(z,\mu)$ converge to $(z^*,0)$ and $x =p_\mu(z)$. Then for each $i\in \mathcal{I}^*$, there exist a positive constant $\kappa_i $  such that 
$$\cfrac{\mu}{A_i x - b_i} > \cfrac{\kappa_i}{\mu}.$$ 
\end{proposition}
See proof of Proposition \ref{prop:limit_polyhedral} in  \ref{proof4}. We are ready to prove that Statement $(i)$ implies Statement $(ii)$ in Theorem \ref{theorem:verify} for polyhedral set.
\begin{proof}
Let $A_{\mathcal{I}^*}$ be the matrix containing the rows $A_i, i\in \mathcal{I}^*$, $N$ be its null space, i.e., $N=\{x: A_i x=0 , \forall i\in \mathcal{I}^* \}$, and $N^\perp = \spn\{ A_i^T, i\in \mathcal{I}^*\}$. We can verify that $\Pi_N(w)=\derD\Pi_{P(A,b)}(Z^*)w$ for $w\in \mathbb{R}^n$ (see Statement $(iv)$ of Proposition \ref{prop:polyhedral_properties}). Let $w$ with $\norm{w}=1$ be fixed and $\der p_\mu(z)w=u$. We recall that $\der p_\mu(z) = [ I + \mu^2 \nabla^2 f(x)]^{-1}$, then we have
\begin{align} \label{temp6}
w=u + \mu^2 \nabla^2 f(x)u = u + \sum\limits_{i=1}^m \cfrac{\mu^2}{(A_ix-b_i)^2} (A_i u) A_i^T.
\end{align}
This implies
\begin{align} \label{temp61}
\begin{split}
\norm{w}^2 & = \norm{u}^2 + 2 \sum\limits_{i=1}^m \cfrac{\mu^2}{(A_ix-b_i)^2} (A_i u)^2 + \norm{\sum\limits_{i=1}^m \cfrac{\mu^2}{(A_ix-b_i)^2} (A_i u) A_i^T}^2\\
 &\geq \max\left\{\norm{u}^2,\sum\limits_{i=1}^m \cfrac{\mu^2}{(A_ix-b_i)^2} (A_i u)^2 \right \}.
\end{split}
\end{align} 
Inequality \eqref{temp61} shows that $u$ is bounded.  Furthermore, we remind that  $x\rightarrow \bar{z}^*$ when $(z,\mu)\rightarrow (z^*,0)$ and $A_i\bar{z}^* > b_i, \forall i\not\in \mathcal{I}^*$.  Therefore, from \eqref{temp6} we deduce
$$\norm{w-u-\sum\limits_{i\in \mathcal{I}^*} \cfrac{\mu^2}{(A_ix-b_i)^2} (A_i u) A_i^T}=\norm{\sum\limits_{i\not\in \mathcal{I}^*} \cfrac{\mu^2}{(A_ix-b_i)^2} (A_i u) A_i^T }= O(\mu).
$$
Moreover, $w=\Pi_N(w) + \Pi_{N^\perp}(w)$ and $\sum\limits_{i\in \mathcal{I}^*} \cfrac{\mu^2}{(A_ix-b_i)^2} (A_i u) A_i^T \in N^\perp$, we hence get
\begin{align} \label{temp62}
\begin{split}
&\dist(\Pi_N(w) - u, N^\perp)\\
 &\leq \dist( w-u-\sum\limits_{i\in \mathcal{I}^*} \cfrac{\mu^2}{(A_ix-b_i)^2} (A_i u) A_i^T,N^\perp) + \dist(-\Pi_{N^\perp}(w),N^\perp)\\
&\qquad +  \dist(\sum\limits_{i\in \mathcal{I}^*} \cfrac{\mu^2}{(A_ix-b_i)^2} (A_i u) A_i^T,N^\perp)\\ 
& =O(\mu).
\end{split}
\end{align}
We have $\Pi_N(u) = u - A_{\mathcal{I}^*}^\dagger A_{\mathcal{I}^*} u,$ where $A_\mathcal{I}^\dagger$ is the Moore-Penrose pseudo-inverse of $A_{\mathcal{I}^*}$. Thus,  
$\dist(u, N) =\norm{A_{\mathcal{I}^*}^\dagger A_{\mathcal{I}^*} u} = O (A_{\mathcal{I}^*}u).$
Moreover, Inequality \eqref{temp61} together with Proposition  \ref{prop:limit_polyhedral} yields that $A_iu=O(\mu) $ for $i\in \mathcal{I}^*$. Therefore,
$$
\dist(u-\Pi_N(w),N)\leq \dist(u,N) + \dist(-\Pi_N(w),N) = O(\mu).
$$
In company with \eqref{temp62}, we imply 
$$
\norm{\Pi_N(w)-u} \leq \norm{\Pi_N(\Pi_N(w)-u)} + \norm{\Pi_{N^\perp}(\Pi_N(w)-u)} = O(\mu)=O(\norm{(z-z^*,\mu)}).
$$
The result follows then. \QEDB
\end{proof} 

\subsubsection{Epigraph of $l_\infty$ norm $C_n$} 
The $l_\infty$ norm cone, which is defined by
$C_n=\{(t,x)\in\mathbb{R}\times \mathbb{R}^n: t\geq \|x\|_\infty\},$
is a special case of polyhedral set since we can rewrite it as
$C_n =\left\{ (t,x)\in\mathbb{R}\times \mathbb{R}^n,\, t-x_i \geq 0,\, t+ x_i \geq 0,\, \forall i=1,\ldots,n\right\}.
$
Then we use the following barrier for the  $l_\infty$ norm cone
 $$f(t,x)=-\sum\limits_{i=1}^n \log(t - x_i)-\sum\limits_{i=1}^n \log(t + x_i).$$
Its gradient and Hessian are
$$\nabla f(t,x)=-\sum\limits_{i=1}^n \cfrac{a_i}{t-x_i}-\sum\limits_{i=1}^n \cfrac{b_i}{t+x_i},
$$
$$\nabla^2 f(t,x)=\sum\limits_{i=1}^n \cfrac{a_ia_i^T}{(t-x_i)^2}+\sum\limits_{i=1}^n \cfrac{b_ib_i^T}{(t+x_i)^2},
$$
where $a_i=e_0-e_i$, $b_i=e_0+e_i$ and  $ e_i, i=0,\ldots, n$, are unit vectors of $\mathbb{R}^{n+1}$. The barrier-based smoothing approximation $p_\mu(z_o,z)=(t,x)$ is then defined by
\begin{align} \label{eq:smooth_equa_vector}
\begin{split}
t-\sum\limits_{i=1}^n \cfrac{2\mu^2}{t^2-x_i^2} t =z_o,\\
 x_i + \cfrac{2\mu^2}{t^2-x_i^2}x_i =z_i,\\
 t> |x_i|, 1\leq i\leq n.
 \end{split}
\end{align}
\begin{remark}
We note that the smoothing approximation in \eqref{eq:smooth_equa_vector} coincides with the smoothing approximation proposed by Chen in his thesis \cite{CHChen}. However, Chen uses a different approach to derive \eqref{eq:smooth_equa_vector}. In particular, he shows that the SA is the unique solution of the logarithmic penalty problem associated with the constrained optimization problem that finds the projection onto the epigraph of $l_\infty$ norm. We derive the barried-based SA for general polyhedral sets in  \eqref{eq:smooth_equa_polyhedral}, and \eqref{eq:smooth_equa_vector} is just a special case of \eqref{eq:smooth_equa_polyhedral}. 
\end{remark}
The following proposition gives another approach other than that of \cite[Proposition 3.2]{Dingetal} to find the projection onto $C_n$ which is used in the next section for epigraph of matrix operator norm. We give its proof in  \ref{proof3}.
\begin{proposition}
\label{prop:proj_vector_cone}
When $(z_o,z,\mu)\rightarrow (z_o^*,z^*,0)$,  the limit of the smoothing approximation defined by \eqref{eq:smooth_equa_vector}is the pair $(t^*,x^*)$ given by 
\begin{align}\label{eq:proj_l_infty}
\begin{split}
t^*(z_o^*,z^*)& = \max \left\{\cfrac{1}{\mathbf k^*+1}(z_o^* + \sum\limits_{i=1}^{\mathbf k^*} | z^*_{\pi(i)}|),0 \right\},\\
x_i^*& = \begin{cases} \sgn(z_i^*) t^* \; &\mbox{for}\; i=\pi(1), \ldots, \pi(\mathbf k^*),\\
z_i^* \; & \mbox{for}\; i=\pi(\mathbf k^*+1), \ldots, \pi(n),\end{cases}
\end{split}
\end{align}
where $\pi$ is a permutation of $\{1,\ldots,n\}$ such that $|z^*_{\pi(1)}| \geq \ldots \geq |z^*_{\pi(n)}|,$ and $\mathbf k^*$ is the unique nonnegative integer satisfying 
$$|z^*_{\pi(\mathbf k^*)}|> \max \left\{\cfrac{1}{\mathbf k^*+1}(z_o^* + \sum\limits_{i=1}^{\mathbf k^*} | z^*_{\pi(i)}|),0 \right\} \geq |z^*_{\pi(\mathbf k^*+1)}|,
$$
where we let $z^*_{\pi(0)}=\infty$ and $z^*_{\pi(n+1)}=0$. Consequently, $\Pi_{C_n}(z_o^*,z^*)=(t^*,x^*)$.
\end{proposition}

\subsection{Epigraph of matrix operator norm and epigraph of matrix nuclear norm} \label{sec:epi_oper_norm}
We recall the self-concordant barrier function used for  $K_{m,n}$ is
\begin{equation} \label{eq:barrierKmn}
f(t,x)=-\logdet\: \begin{pmatrix} & tI_n & x^T\\&x &tI_m\end{pmatrix}.
\end{equation}
Its first derivative is 
\begin{align*}
\begin{split}
&\derD f(t,x)[\triangle t, \triangle x]\\
&= -\left \langle  \begin{pmatrix} & tI_n & x^T\\&x &tI_m\end{pmatrix}^{-1},\begin{pmatrix} \triangle tI_n &\triangle x^T\\ \triangle x & \triangle tI_m
\end{pmatrix} \right \rangle \\
&=\left \langle \begin{pmatrix}
& (tI_n - \frac{1}{t}x^Tx)^{-1} & -\frac{1}{t} x^T(t I_m- \frac{1}{t}xx^T )^{-1}\\
& - \frac{1}{t} (tI_m -\frac{1}{t}xx^T)^{-1} x & (t I_m - \frac{1}{t}xx^T)^{-1}
\end{pmatrix},\begin{pmatrix} \triangle tI_n &\triangle x^T\\ \triangle x & \triangle tI_m
\end{pmatrix} \right \rangle.
\end{split}
\end{align*}
This implies $$\nabla f(t,x)=\left(- \trace(tI_n - \frac{1}{t}x^Tx)^{-1} - \trace(t I_m - \frac{1}{t}xx^T)^{-1}, \frac{2}{t}(tI_m -\frac{1}{t}xx^T)^{-1} x \right).$$
Denote $\Sigma = [\Diag(\sigma_f(x))\quad 0]$, and let $x= u \Sigma v^T$. It follows from
$$(tI_n - \frac{1}{t}x^Tx) =v (t I_n -\frac{1}{t}\Sigma^T \Sigma ) v^T\text{and}\; (tI_m - \frac{1}{t}xx^T) = u (t I_m -\frac{1}{t}\Sigma \Sigma^T ) u^T
$$
that $\nabla f(t,x)=\left(-\sum\limits_{i=1}^m \cfrac{2t}{t^2 - \sigma_i(x)^2} -\cfrac{n-m}{t}\;,\; 2 u ( t^2 I_m - \Sigma \Sigma^T)^{-1} \Sigma v^T\right).$
The equation that defines the corresponding barrier-based smoothing approximation $P_\mu(z_o,z)$ of the projection onto $K_{m,n}$, $(t,x) + \mu^2 \nabla f(t,x) = (z_o,z)
$,
is rewritten as
\begin{equation} \label{eq:smooth_equ_matrix}
\begin{cases}
 t -\mu^2\left(\sum\limits_{i=1}^m \cfrac{2t}{t^2 - \sigma_i^2} + \cfrac{n-m}{t}\right) = z_o,   \\
\sigma_i + 2\mu^2 \cfrac{\sigma_i}{t^2 - \sigma_i^2}, = \sigma_i^o \\
t>|\sigma_i|,
\end{cases}
\end{equation}
where $\sigma =\sigma_f(x), \sigma^o=\sigma_f(z), z = u[\Diag(\sigma^o)\quad  0 ] v^T $. 

For $K_{m,n}^\sharp$, using the modified Fenchel barrier function gives us the corresponding smoothing approximation $P_\mu^\sharp(z_o,z)$. From \cite[Part 6.2]{Chua_Hien}, we have 
\begin{align} \label{eq:smthKmn}
P_\mu^\sharp(z_o,z) = P_\mu(-z_o,-z) + (z_o,z).
\end{align}  
Now we give characteristic of the projector onto $K_{m,n}$. For $(t,x)\in \mathbb{R}\times \mathbb{R}^{m\times n}$, we let $x=u[\Diag(\sigma_f(x)) \, 0]v^T$ be a singular value decomposition of $x$, and denote 
$$a=\{i: \sigma_i(x)>0\},\, b=\{ i: \sigma_i(x)=0\},\, c=\{ m+1,\ldots,n\},\, c'=\{1,\ldots,n-m\},
$$
$$(q_0(t,\sigma(x)),q(t,\sigma(x)))= \Pi_{C_n}(t,\sigma(x)),
$$  
and $\Omega_1, \Omega_2 \in \mathbb{R}^{m\times m}, \Omega_3\in \mathbb{R}^{m\times (n-m)}$ as follows
$$ (\Omega_1) _{ij}= \left\{ \begin{array}{lll}
\cfrac{q_i(t,\sigma(x))-q_j(t,\sigma(x))}{\sigma_i(x)-\sigma_j(x)} &\text{if}\; \sigma_i(x)\ne \sigma_j(x), \\ 
0 &\text{otherwise},
\end{array} \right. \rm{for} \, i,j \in \{1,\ldots,m\}
$$
$$(\Omega_2)_{ij}= \left\{ \begin{array}{lll}
\cfrac{q_i(t,\sigma(x))+ q_j(t,\sigma(x))}{\sigma_i(x)+\sigma_j(x)} &\text{if}\; \sigma_i(x)+ \sigma_j(x) \ne 0,\\ 
0 &\text{otherwise},
\end{array} \right.  \rm{for} \, i,j \in \{1,\ldots,m\}
$$
$$(\Omega_3)_{ij} = \left\{ \begin{array}{lll}
\cfrac{q_i(t,\sigma(x))}{\sigma_i(x)} &\text{if}\; \sigma_i(x)\ne 0, \\ 
0 &\text{otherwise},
\end{array} \right.   \rm{for} \, i\in \{1,\ldots,m\}, j\in\{ 1,\ldots,n-m\}.
$$
We can rewrite these matrices as follows
$$ \Omega_1 =\begin{bmatrix}
& 0&0 &(\Omega_1)_{\alpha\gamma}\\
& 0 & 0& E_{\beta \gamma}\\
&(\Omega_1)_{\gamma\alpha} & E_{\gamma\beta} &(\Omega_1)_{\gamma\gamma}
\end{bmatrix}, \Omega_2=\begin{bmatrix}
&(\Omega_2)_{aa} & (\Omega_2)_{ab}\\
&(\Omega_2)_{ba} & 0
\end{bmatrix}, \Omega_3=\begin{bmatrix}
&(\Omega_3)_{ac'}\\
&0
\end{bmatrix},
$$
where $E_{\beta\gamma}$, $E_{\gamma\beta}$ are two matrices whose entries are all ones, and 
$$\alpha=\{ i: \sigma_i > q_0(t,\sigma(x))\}, \beta=\{ i: \sigma_i= q_0(t,\sigma(x))\}, \gamma=\{ i: \sigma_i < q_0(t,\sigma(x))\}. $$ 
Let $\bar{\mathbf k}$ be number of $q_i(t,\sigma(x))$ such that $q_i(t,\sigma(x))=q_0(t,\sigma(x))$. Denote $$\delta=\sqrt{1+\bar{\mathbf k}},\; \rho(w_o,w)=\begin{cases} \delta^{-1} (w_o+Tr(\mathfrak S(u_\alpha^T w v_\alpha)) ) &\mbox{if}\; t\geq \|x\|_*, \\ 0& \mbox{otherwise}. \end{cases}$$

The following theorem characterizes the differentiable property of the projection onto $K_{m,n}$.
\begin{theorem} \cite[Theorem 3]{Dingetal} \label{theorem:diff_proj_oper_norm}
The metric projector $\Pi_{K_{m,n}}$ is differentiable at $(t,x)$ if and only if $(t,x)$ satisfies one of the following conditions
\begin{itemize}
\item[(i)] $t> \|x\|_2$, 
\item[(ii)] $\|x\|_2 > t > -\|x\|_*$ but $\beta=\emptyset$, 
\item[(iii)] $t<-\|x\|_*$.
\end{itemize}
Under condition (ii), we have $\derD \Pi_{K_{m,n}}(t,x)(w_o,w)=(w_o',w')$, where
$$(w_o,w) \in \mathbb R \times \mathbb{R}^{m\times n}, \quad w_o'=\delta^{-1} \rho(w_o,w)$$
and 
$$\begin{array}{ll} w'&= u \begin{bmatrix}
& \delta^{-1} \rho(w_o,w) I_{|\alpha|} & (\Omega_1)_{\alpha\gamma}\circ \mathfrak S(A)_{\alpha\gamma}\\
&(\Omega_1)_{\gamma\alpha} \circ \mathfrak S(A)_{\gamma\alpha} & \mathfrak S(A)_{\gamma\gamma}
\end{bmatrix} v_1^T \\ 
 &+ u \begin{bmatrix}
&(\Omega_2)_{aa} \circ \mathfrak T(A)_{aa} & (\Omega_2)_{ab}\circ \mathfrak T(A)_{ab} \\
& (\Omega_2)_{ba} \circ \mathfrak T(A)_{ba} & \mathfrak T(A)_{bb}
\end{bmatrix} v_1^T + u \begin{bmatrix}
(\Omega_3)_{ac'} \circ B_{ac'} \\
B_{bc'}
\end{bmatrix} v_2^T
\end{array}$$
with $v=[v_1|v_2]$, $A=u^T w v_1$, and $B =u^T w v_2$.
\end{theorem}
We need the following lemmas to prove the convergence of the derivative $\derD p_\mu(z_o,z)$.
\begin{lemma} \label{lemma2} \cite[Lemma 3]{Chen_Tseng}
For any symmetric matrix $x\in \mathbb{S}^n$, there exist $\eta>0$ and $\varepsilon>0$ such that $
\min\limits_{p\in \mathcal{O}_x} \norm{p-q} \leq \eta \norm{x-y}, \, \forall\, y$, and $\norm{y-x}\leq \varepsilon, \forall\, q\in \mathcal{O}_y.
$\end{lemma}
From $\sigma_k(z)=\min\limits_{\mathrm{rank}(y) < k} \norm{z-y}$ (see \cite[Chapter III]{Bhatia}), we can derive the following lemma.
\begin{lemma}\label{lemma3}  The $k$-th singular value $\sigma_k(\cdot)$ of a matrix in $\mathbb{R}^{m\times n}$ satisfied 
$$ \left|\sigma_k(z_1)- \sigma_k(z_2)\right|\leq \norm{z_1 - z_2} \,\text{for all matrices}\,  z_1,z_2 \,\text{in}\, \mathbb{R}^{m\times n}.$$ 
\end{lemma}

\subsubsection{Verifying the requirement of Theorem \ref{theorem:verify} (ii)}
Let  $(z_o^*,z^*)$ be a differentiable point of the projector onto $K_{m,n}$ ( or $K_{m,n}^\sharp$), and let $(z_o, z, \mu)$ go to $(z_o^*,z^*,0)$.  We now verify the following expression which is the requirement of Theorem \ref{theorem:verify} (ii)
\begin{equation} \label{eq:checkKmn}
\norm{\derD P_\mu(z_o,z)-\derD \Pi_K(z_o^*,z^*)} = O(\norm{(z_o-z_o^*,z-z^*,\mu)}). 
\end{equation}
\begin{proof} 
Let $(t,x) = p_{\mu}(z_o,z)$ with $x=u[\Diag(\sigma)\quad 0] v^T$, $z=u[ \Diag(\sigma^o) \quad 0] v^T$ and  $(t,\sigma)$, $(z_o,\sigma^o)$ satisfying \eqref{eq:smooth_equ_matrix}.  Let $u^*, v^*$ be the limit points of $u,v$. When $\mu\to 0$, we have $\sigma^o \to \sigma_f(z^*)=\zeta$ and  $z \to z^* = u^*[ \Diag(\zeta) \quad 0] (v^*)^T$. Totally similar to Proposition \ref{prop:proj_vector_cone}, we can prove $(t,\sigma) \rightarrow (t^*,\sigma^*) = \Pi_{C_n}(z_o^*,\zeta)$, where
$$ \sigma_i^* = \begin{cases} 
t^* > 0 &\mbox{if}\; i\leq \mathbf k^*,\\ 
\zeta_i < t^* &\mbox{if}\; i> \mathbf k^*.
\end{cases}
$$
Then $x \rightarrow x^*=u^*[\Diag(\sigma^*) \quad 0] (v^*)^T$.
The second derivative of the barrier \eqref{eq:barrierKmn} is

\begin{align*}
\derD^2 f(t,x): ((\mathsf h_o,\mathsf h); (\mathsf k_o, \mathsf k)) \mapsto \trace\begin{pmatrix}
tI_n & x^T\\
x &t I_m
\end{pmatrix}^{-1} 
\begin{pmatrix}
\mathsf h_o I_n & \mathsf h^T\\
\mathsf h& \mathsf h_oI_m
\end{pmatrix} 
\begin{pmatrix}
tI_n&x^T\\
x & tI_m
\end{pmatrix}^{-1} 
\begin{pmatrix}
\mathsf k_o I_n & \mathsf k^T \\
\mathsf k & \mathsf k_o I_m
\end{pmatrix}.
\end{align*}

We note that 
\begin{align*}
\begin{split}
\begin{pmatrix}
tI_n & x^T\\
x &t I_m
\end{pmatrix}^{-1}&=\begin{pmatrix}
v&0\\0&u
\end{pmatrix} 
\begin{pmatrix}
t(t^2 I_n - \Sigma^T \Sigma)^{-1} & -\Sigma^T(t^2 I_n - \Sigma\Sigma^T)^{-1}\\
-(t^2 I_m - \Sigma\Sigma^T)^{-1}\Sigma & t(t^2 I_m -\Sigma\Sigma^T)^{-1}
\end{pmatrix}
\begin{pmatrix}
v^T&0\\0&u^T
\end{pmatrix}\\
&=\begin{pmatrix}
v_1&v_2&0\\
0&0&u
\end{pmatrix} \begin{pmatrix}
D_1 &0& \Delta\\
0&D_2&0\\
\Delta&0&D_1
\end{pmatrix} 
\begin{pmatrix}
v_1^T & 0 \\
v_2^T & 0 \\
0 &u^T
\end{pmatrix},
\end{split}
\end{align*}
where $D_1=\Diag\bracket{\left. \cfrac{t}{t^2-\sigma_i^2}\right|_{i=1,\ldots,m}}, D_2=\cfrac{1}{t}I_{n-m},\Delta=\bracket{\left.\cfrac{-\sigma_i}{t^2 -\sigma_i^2} \right|_{i=1,\ldots,m}}$, $v_1$ contains the first $m$ columns of $v$ and $v_2$ contains the remaining $n-m$ columns of $v$. Then, 
\begin{align*}
\begin{split}
&\derD^2 f(t,x)[(\mathsf h_o,\mathsf h); (\mathsf k_o, \mathsf k)]\\
&=\trace\begin{pmatrix}
D_1&0&\Delta\\0&D_2&0\\\Delta&0&D_1
\end{pmatrix} 
\begin{pmatrix}
\mathsf h_o I_m&0&\bar{\mathsf h}_1^T\\ 0 &\mathsf h_o I_{n-m}&\bar{\mathsf h}_2^T\\\bar{\mathsf h}_1 & \bar{\mathsf h}_2 & \mathsf h_o I_m
\end{pmatrix} 
\begin{pmatrix}
D_1&0&\Delta\\0&D_2&0\\\Delta&0&D_1
\end{pmatrix}
\begin{pmatrix}
\mathsf k_oI_m&0&\bar{\mathsf k}_1^T\\ 0 &h_o I_{n-m}&\bar{\mathsf k}_2^T\\\bar{\mathsf k}_1 & \bar{\mathsf k}_2 & \mathsf k_o I_m
\end{pmatrix}\\
&=\trace\begin{pmatrix}
\mathsf h_o(D_1^2+\Delta^2) + \Delta\bar{\mathsf h}_1D_1+D_1\bar{\mathsf h}_1^T\Delta &\Delta\bar{\mathsf h}_2D_2&2\mathsf h_o\Delta D_1+D_1\bar{\mathsf h}_1^TD_1+\Delta\bar{\mathsf h}_1 \Delta\\
D_2\bar{\mathsf h}_2^T\Delta & \mathsf h_oD_2^2 &D_2\bar{\mathsf h}_2^TD_1\\
2 \mathsf h_o\Delta D_1+D_1\bar{\mathsf h}_1D_1+\Delta\bar{\mathsf h}_1^T\Delta & D_1\bar{H}_2 D_2 & \mathsf h_o(D_1^2+\Delta^2)+\Delta\bar{\mathsf h}_1^T D_1+ D_1 \bar{\mathsf h}_1 \Delta
\end{pmatrix} \\
&\qquad\qquad\times\begin{pmatrix}
\mathsf k_o I_m &0 &\bar{\mathsf k}_1^T \\ 0& \mathsf k_o I_{n-m} & \bar{\mathsf k}_2^T\\ \bar{K}_1 & \bar{\mathsf k}_2 & \mathsf k_o I_m
\end{pmatrix}.
\end{split}
\end{align*}
Hence we get
\begin{align}\label{temp81}
\begin{split}
&\derD^2 f(t,x)[(\mathsf h_o,\mathsf h); (\mathsf k_o, \mathsf k)]=\mathsf k_o\bracket{\mathsf h_o(2\trace(D_1^2 + \Delta^2) + \norm{D_2}_F^2) + 4\trace (\Delta \bar{\mathsf h}_1D_1)}  \\ 
&\qquad\qquad\qquad+ 2\langle \bar{\mathsf k}_1, 2 \mathsf h_o \Delta D_1 + D_1 \bar{\mathsf h}_1 D_1 + \Delta \bar{\mathsf h}_1^T\Delta \rangle+ 2\langle \bar{\mathsf k}_2, D_1 \bar{\mathsf h}_2 D_2\rangle,
\end{split}
\end{align}
where $\bar{\mathsf h}_i = u^T \mathsf h v_i$ and $\bar{\mathsf k}_i=u^T \mathsf k v_i$, $i =1,2$. 

Let $(w_o,w)\in \mathbb{R}\times\mathbb{R}^{m\times n}$ be fixed. If $\derD P_\mu(z_o,z)[w_o,w]=(\mathsf h_o,\mathsf h)$, then $$(w_o,w)=(\mathsf h_o,\mathsf h)+\mu^2 \derD^2 f(t,x)[\mathsf h_o,\mathsf h],$$ and hence for any $(\mathsf k_o,\mathsf k) \in \mathbb{R}\times\mathbb{R}^{m\times n}$, by formula \eqref{temp81}, we have
\begin{align*}
\begin{split} 
&\mathsf k_o(w_o-\mathsf h_o) + \langle \bar{\mathsf k}_1, \bar{w}_1 - \bar{\mathsf h}_1\rangle + \langle \bar{\mathsf k}_2, \bar{w}_2 - \bar{\mathsf h}_2\rangle\\ &= \langle (w_o-\mathsf h_o, w-\mathsf h) , (\mathsf k_o,\mathsf k)\rangle =\mu^2 \derD^2 f(t,x)[(\mathsf h_o,\mathsf h); (\mathsf k_o,\mathsf k)]\\
&=\mu^2 \mathsf k_o\bracket{\mathsf h_o(2\trace(D_1^2 + \Delta^2) + \norm{D_2}_F^2) + 4\trace (\Delta \bar{\mathsf h}_1D_1)} \\
&\qquad\qquad\qquad+ 2\mu^2\langle \bar{\mathsf k}_1, 2\mathsf h_o \Delta D_1 + D_1 \bar{\mathsf h}_1 D_1 + \Delta \bar{\mathsf h}_1^T \Delta\rangle + 2 \mu^2\langle \bar{\mathsf k}_2, D_1 \bar{\mathsf h}_2 D_2\rangle,
\end{split}
\end{align*} 
where  $\bar{w}_i=u^T w v_i$, $i =1,2$.
Thus for every fixed $(w_o,w)$, we get 
\begin{align} \label{eq:der_Kmn}
\begin{split}
w_o - \mathsf h_o &= \mu^2 \bracket{\mathsf h_o(2\trace(D_1^2 + \Delta^2) + \norm{D_2}_F^2) + 4\trace (\Delta \bar{\mathsf h}_1D_1)},\\
\bar{w}_1 - \bar{\mathsf h}_1& = 2\mu^2 \bracket{2 \mathsf h_o \Delta D_1 + D_1 \bar{\mathsf h}_1 D_1 + \Delta (\bar{\mathsf h}_1)^T \Delta},\\
\bar{w}_2 - \bar{\mathsf h}_2 &= 2\mu^2 D_1 \bar{\mathsf h}_2 D_2.
\end{split}
\end{align}
By Theorem \ref{theorem:diff_proj_oper_norm}, the Euclidean projector $\Pi_{K_{m,n}}$ is differentiable at $(z_o^*,z^*)$ if and only if
\begin{itemize}
\item[(i)] $z_o^*> \|z^*\|$,
\item[(ii)] $\|z^*\|_2 > z_o^* > -\|z^*\|_*$ but $t^*(z_o^*,\sigma_f(z^*))$, which is defined in \eqref{eq:proj_l_infty}, is not a singular value of $z^*$,
\item[(iii)] $z_o^*<-\|z^*\|_*$.
\end{itemize}

We consider the first case $z_o^* > \| z^*\|$, i.e., $(z_o^*,z^*)$ lies in the interior of $K_{m,n}$. In this case $(t,x) \to \Pi_{K_{m,n}}(z_o^*,z^*)=(z_o^*,z^*)$ and $\derD \Pi_{K_{m,n}}(z_o^*,z^*)=\mathbf I$. Thus, 
$$\norm{(\mathsf h_o,\mathsf h)-(w_o,w)} = \norm{\mu^2 \nabla^2 f(t,x)(\mathsf h_o,\mathsf h)} = O(\mu)=O(\norm{(z_o-z_o^*,z-z^*,\mu)}).
$$
We now consider the second case. Denote $\bar{w}^*=(u^*)^T w v^*$. From the third equation of \eqref{eq:der_Kmn}, we imply that for $i=1,\ldots,m$, $j=1,\ldots,n-m$, we have
$\bar{w}_{i, m+j} -\bar{\mathsf h}_{i,m+j} = 2\cfrac{\mu^2}{t^2 - \sigma_i^2} \bar{\mathsf h}_{i,m+j}$. Therefore, we get $\bar{\mathsf  h}_{i,m+j} = \cfrac{1}{1+ 2\frac{\mu^2}{t^2 - \sigma_i^2}} \bar{w}_{i,m+j}$. Furthermore, from \eqref{eq:smooth_equ_matrix}, we deduce 
\begin{align}\label{limit:l1}
\begin{split}
\cfrac{t^2 - \sigma_i^2}{\mu^2}=\cfrac{2\sigma_i}{\sigma_i^o -\sigma_i} \rightarrow \cfrac{2t^*}{\zeta_i-t^*} \;\text{for}\; i=1,\ldots,\mathbf k^*, \; \text{and} \\
\cfrac{\mu^2}{t^2 - \sigma_i^2} \rightarrow \cfrac{0}{(t^*)^2 - (\zeta_i)^2 } =0 \; \text{for}\, i=\mathbf k^*+1,\ldots,m.
\end{split}
\end{align} 
Hence, for $i=1,\ldots,m$, $j=1,\ldots,n-m$, we have
\begin{equation}\label{limit:l2}
\bar{\mathsf h}_{i,m+j} \rightarrow \begin{cases} 
\cfrac{t^*}{\zeta_i}\bar{w}^*_{i,m+j} &\mbox{if} \; i =1,\ldots,\mathbf k^*\\ \bar{w}^*_{i,m+j} &\mbox{if}\; i=\mathbf k^*+1,\ldots,m. \end{cases}
\end{equation}
For $i,j=1,\ldots,m, i\ne j$,  denote $\rho_{ij}=\cfrac{t^2 -\sigma_i^2}{2\mu^2} \times (t^2 - \sigma_j^2)$. The $(i,j)$-th, $(j,i)$-th entries of the second equation of \eqref{eq:der_Kmn} give
$\bar{w}_{ij}-\bar{\mathsf h}_{ij}=\rho_{ij}^{-1}(\bar{\mathsf h}_{ij} t^2 + \bar{\mathsf h}_{ji} \sigma_i \sigma_j)$ and $\bar{w}_{ji}-\bar{\mathsf h}_{ji}=\rho_{ij}^{-1}(\bar{\mathsf h}_{ji} t^2 + \bar{\mathsf h}_{ij} \sigma_i \sigma_j).
$ Solving these equations imply that for $i,j=1,\ldots,m, i\ne j$, we have
$$
\bar{\mathsf h}_{ij}= \cfrac{\rho_{ij}}{(\rho_{ij}+t^2)^2 -\sigma_i^2 \sigma_j^2}(\bar{w}_{ij}(\rho_{ij} + t^2) - \bar{w}_{ji}\sigma_i \sigma_j),
\bar{\mathsf h}_{ji}= \cfrac{\rho_{ij}}{(\rho_{ij}+t^2)^2 -\sigma_i^2 \sigma_j^2}(\bar{w}_{ji}(\rho_{ij} + t^2) - \bar{w}_{ij}\sigma_i \sigma_j).$$
For $i=1,\ldots,\mathbf k^*, j=\mathbf k^*,\ldots,m$, we have
\begin{align} \label{limit:l3}
\rho_{ij} \rightarrow \cfrac{t^*((t^*)^2-(\zeta_j)^2)}{\zeta_i-t^*}=\rho_{ij}^*;
\end{align}
hence
\begin{align}\label{limit:l4}
\bar{\mathsf h}_{ij}\to \cfrac{\rho_{ij}^* }{(\rho_{ij}^* + (t^*)^2)^2 -(t^*)^2(\zeta_j)^2} (\bar{w}^*_{ij}(\rho_{ij}^* + (t^*)^2) - \bar{w}^*_{ji}t^* \zeta_j)= \cfrac{t^*\zeta_i - \zeta_j^2}{\zeta_i^2 - \zeta_j^2 }\bar{w}^*_{ij} - \cfrac{\zeta_j(\zeta_i - t^*)}{\zeta_i^2 - \zeta_j^2}\bar{w}^*_{ji},
\end{align}
 \begin{align} \label{limit:l5}
\bar{\mathsf h}_{ji}\rightarrow   \cfrac{t^*\zeta_i - \zeta_j^2}{\zeta_i^2 - \zeta_j^2 }\bar{w}^*_{ji} - \cfrac{\zeta_j(\zeta_i - t^*)}{\zeta_i^2 - \zeta_j^2}\bar{w}^*_{ij}.
\end{align}
For $i,j=1,\ldots,\mathbf k^*, i\ne j$, we have
\begin{align} \label{limit:l51}
\rho_{ij}\to 0, \quad \frac{\rho_{ij}}{2\mu^2}=\cfrac{t^2 -\sigma_i^2}{2\mu^2} \times\cfrac{ (t^2 - \sigma_j^2)}{2\mu^2} \rightarrow \frac{(t^*)^2}{(\zeta_i-t^*)(\zeta_j-t^*)};
\end{align}
\begin{align} \label{l52}
\begin{split}
\bar{\mathsf h}_{ij}&=\cfrac{\frac{\rho_{ij}}{2\mu^2}}{\rho_{ij} \frac{\rho_{ij}}{2\mu^2} + 2 t^2 \frac{\rho_{ij}}{2\mu^2}+ t^2 \frac{t^2-\sigma_i^2}{2\mu^2} + \sigma_i^2 \frac{t^2 - \sigma_j^2}{2\mu^2}} (\bar{w}_{ij}(\rho_{ij} + t^2) - \bar{w}_{ji}\sigma_i \sigma_j)\to  \frac{t^*}{\zeta_i + \zeta_j} (\bar{w}^*_{ij} - \bar{w}^*_{ji}).
\end{split}
\end{align}
For  $i=\mathbf k^*+1,\ldots,m$, the $i$-th diagonal entry of the second equation in \eqref{eq:der_Kmn} is 
\begin{align}\label{limit:l6}
\bar{w}_{ii} - \bar{\mathsf h}_{ii} = 2\bracket{\cfrac{\mu}{t^2 - \sigma_i^2}}^2\bracket{-2 \mathsf h_o t \sigma_i + \bar{\mathsf h}_{ii}(t^2+\sigma_i^2)} \rightarrow 0,
\end{align}
which shows that $\bar{\mathsf h}_{ii}\rightarrow \bar{w}^*_{ii}$. Similarly, 
\begin{align} \label{limit:l7}
 \bar{\mathsf h}_{ij}\rightarrow \bar{w}^*_{ij}\; \text{for}\; i,j=\mathbf k^*+1,\ldots,m, i \ne j. 
\end{align}
For  $i=1,\ldots,\mathbf k^*$, the $i$-th diagonal entry of the second equation in \eqref{eq:der_Kmn} is 
$$\bar{w}_{ii} - \bar{\mathsf h}_{ii} = 2\bracket{\cfrac{\mu}{t^2 - \sigma_i^2}}^2\bracket{-2 \mathsf h_o t \sigma_i + \bar{\mathsf h}_{ii}(t^2+\sigma_i^2)},
$$
which implies
\begin{align} \label{limit:l8}
 2t^2(\bar{\mathsf h}_{ii} - \mathsf h_o)=\cfrac12 \bracket{\cfrac{t^2 - \sigma_i^2}{\mu^2}}^2 \mu^2 \bracket{\bar{w}_{ii} - \bar{\mathsf h}_{ii}} + 2 \mathsf h_o t (\sigma_i -t) + \bar{\mathsf h}_{ii} (t^2 -\sigma_j^2) \rightarrow 0.
\end{align}
Therefore, 
\begin{equation} \label{limit:l9}
\bar{\mathsf h}_{ii} - \mathsf h_o\rightarrow 0\; \text{for}\; i=1,\ldots,\mathbf k^*.
\end{equation}
Adding the first equation to the sum of the diagonal entries in the second equation of \eqref{eq:der_Kmn},
\begin{align}\label{limit:l11}
\begin{split}
&w_o-\mathsf h_o + \sum\limits_{i=1}^m (\bar{w}_{ii} - \bar{\mathsf h}_{ii}) \\
&= \mu^2 \mathsf h_o \cfrac{n-m}{t} + 2 \sum\limits_{i=1}^m\bracket{\cfrac{\mu}{t^2 - \sigma_i^2}}^2 \bracket{\mathsf h_o(t^2 + \sigma_i^2 - 2t \sigma_i) - \bar{\mathsf h}_{ii}(2t\sigma_i - t^2 -\sigma_i^2)}\\
&=\mu^2 \mathsf h_o \cfrac{n-m}{t^2} + 2\sum\limits_{i=1}^m \bracket{\cfrac{\mu}{t+\sigma
_i}}^2(\mathsf h_o - \bar{\mathsf h}_{ii}) \rightarrow 0.
\end{split}
\end{align}
Thus $\mathsf h_o + \sum\limits_{i=1}^m \bar{\mathsf h}_{ii} \rightarrow w_o + \sum\limits_{i=1}^m \bar{w}^*_{ii}.$ Together with $\bar{\mathsf h}_{ii} \rightarrow \bar{w}^*_{ii}$ for $i=\mathbf k^*+1,\ldots,m$ and $\bar{\mathsf h}_{ii} - \mathsf h_o \rightarrow 0$ for $i=1,\ldots,\mathbf k^*$, we conclude that for $i=1,\ldots,\mathbf k^*$,
\begin{align} \label{limit:l12}
(\mathbf k^*+1) \mathsf h_o \rightarrow w_o + \sum\limits_{i=1}^{\mathbf k^*} \bar{w}^*_{ii},\quad  \bar{\mathsf h}_{ii} \rightarrow \cfrac{1}{\mathbf k^*+1} \bracket{w_o + \sum\limits_{i=1}^{\mathbf k^*} \bar{w}^*_{ii}}.
\end{align}
In summary, 
$\mathsf h_o \rightarrow \mathsf h_o^* =\cfrac{1}{\mathbf k^*+1}\bracket{w_o + \sum\limits_{i=1}^{\mathbf k^*} \bar{w}^*_{ii}},
$
and 
$$\bar{\mathsf h}_{ij}\rightarrow \bar{\mathsf h}^*_{ij}=\begin{cases}
\cfrac{1}{\mathbf k^*+1} \bracket{w_o + \sum\limits_{i=1}^{\mathbf k^*} \bar{w}^*_{ii}} &\mbox{if}\; i,j=1,\ldots,\mathbf k^*,i=j,\\
\cfrac{t^*}{\zeta_i + \zeta_j} (\bar{w}^*_{ij} - \bar{w}^*_{ji}) &\mbox{if}\; i,j=1,\ldots,\mathbf k^*, i\ne j,\\
\cfrac{t^*\zeta_i - \zeta_j^2}{\zeta_i^2 - \zeta_j^2 }\bar{w}^*_{ij} - \cfrac{\zeta_j(\zeta_i - t^*)}{\zeta_i^2 - \zeta_j^2}\bar{w}^*_{ji} & \mbox{if}\; i=1,\ldots,\mathbf k^*,\mbox{and}\; j=\mathbf k^*+1,\ldots,m\\ 
\cfrac{t^*\zeta_j - \zeta_i^2}{\zeta_j^2 - \zeta_i^2 }\bar{w}^*_{ij} - \cfrac{\zeta_i(\zeta_j - t^*)}{\zeta_j^2 - \zeta_i^2}\bar{w}^*_{ji} &\mbox{if}\; i=\mathbf k^*+1,\ldots,m,\mbox{and}\; j=1,\ldots,\mathbf k^*\\
\bar{w}^*_{ij}&\mbox{if}\; i=\mathbf k^*+1,\ldots,m,j=\mathbf k^*+1,\ldots,n\\
\cfrac{t^*}{\zeta_i} \bar{w}^*_{ij}&\mbox{if}\; i=1,\ldots,\mathbf k^*, j= m+1,\ldots, n.
\end{cases}
$$
Therefore, a limit of $\derD P_\mu(z_o,z)$ has the form $T^*(w_o,w)=(\mathsf h_o^*, u^* \bar{\mathsf h}^* (v^*)^T)$, which can be verified to be the derivative of the projector onto $K_{m,n}$ at $(z_o^*,z^*)$  by Theorem \ref{theorem:diff_proj_oper_norm}. 
 
Lipschitz continuity of a smoothing approximation with respect to $\mu$ ( see Theorem \ref{theorem:bsa}) and that of the projector imply 
\begin{align*}
 \norm{(t,\sigma) - (t^*,\sigma^*)}& = \norm{p_\mu(z_o,\sigma^o)-\Pi_{K_{m,n}}(z_o^*,\zeta)}\\
&\leq\norm{p_\mu(z_o,\sigma^o)-p_0(z_o,\sigma^o)}+\norm{\Pi_{K_{m,n}}(z_o,\sigma^o)-\Pi_{K_{m,n}}(z_o^*,\zeta)}\\
&= O(\norm{(z_o-z_o^*,\sigma^o-\zeta,\mu)}.
\end{align*} 
Furthermore, by Lemma \ref{lemma3} we have $\norm{\sigma^o - \zeta}=O(\norm{z-z^*})$. Therefore, the limits in \eqref{limit:l1} satisfy
$$ \frac{\sigma_i}{\sigma_i^o - \sigma_i} -\frac{t^*}{\zeta_i - t^*} = \frac{(\sigma_i -t^*)\zeta_i + t^*(\zeta_i - \sigma_i^o)}{(\sigma_i^o-\sigma_i)(\zeta_i -t^*)} = O(\norm{(z_o-z_o^*,z-z^*,\mu)}), \rm{for} \, i=1,\ldots,\mathbf k^*, \text{and}
$$ 
  $$ \frac{\mu^2}{t^2 - \sigma_i^2} = O(\mu^2)= O(\norm{(z_o-z_o^*,z-z^*,\mu)}), i=\mathbf k^*+1,\ldots,m.
$$
Moreover, using Lemma \ref{lemma2} we deduce there exist $\eta_1,\eta_2, \varepsilon_1, \varepsilon_2 >0$ such that 
$$ \forall z, \norm{z-z^*}<\varepsilon_1, \forall u \in \mathcal{O}(zz^T), \exists u^* \in \mathcal{O}(z^*(z^*)^T)\; \text{such that}\; \norm{u-u^*} \leq \eta_1\norm{z-z^*},$$ 
$$ \forall z, \norm{z-z^*}<\varepsilon_2, \forall v \in \mathcal{O}(z^Tz), \exists v^* \in \mathcal{O}((z^*)^Tz^*)\; \text{such that}\; \norm{v-v^*} \leq \eta_2\norm{z-z^*}.$$ 
In company with the fact $T^*$ is independent of the choice $u^*,v^*$, we can choose $u^*,v^*$ such that $ \norm{\bar{w}-\bar{w}^*}= O(\norm{(z_o-z_o^*,z-z^*,\mu)})$. 
Therefore, the limit in \eqref{limit:l2} satisfy
\[ \begin{cases}  \norm{\bar{h}_{i,m+j}-\cfrac{t^*}{\zeta_i}\bar{w}^*_{i,m+j} }= O(\norm{(z_o-z_o^*,z-z^*,\mu)}) &\mbox{if} \; i =1,\ldots,\mathbf k^*\\
\norm{\bar{h}_{i,m+j}- \bar{w}^*_{i,m+j}}= O(\norm{(z_o-z_o^*,z-z^*,\mu)})&\mbox{if}\; i=\mathbf k^*+1,\ldots,m. \end{cases}
\]
Totally similarly, we can prove that all of the involving limits to finding limits of $\mathsf h_o,\bar{\mathsf h}$ in  \eqref{limit:l3}--\eqref{limit:l12}, which have the form $lhs \to rhs$, satisfy $\norm{lhs-rhs}= O(\norm{(z_o-z_o^*,z-z^*,\mu)})$. This leads to $ \norm{\bar{\mathsf h} -\bar{\mathsf h}^*} = O(\norm{(z_o-z_o^*,z-z^*,\mu)}.$
We then get Expression \eqref{eq:checkKmn}.

Now we consider case $z_o^* < - \norm{z^*}_*$, i.e., $-(z_o^*,z^*)\in\inter K_{(m,n)}^\sharp$. We deduce from \eqref{eq:smthKmn} that $ \derD P^\sharp_\mu(z_o,z) = \mathbf I - \derD P_\mu(-z_o,-z).$
Furthermore,  it follows from $(-z_o,-z)\to (-z_o^*,-z^*)\in \inter (K^\sharp)$ that 
$P^\sharp_\mu(-z_o,-z) \to \Pi_{K^\sharp}(-z_o^*,-z^*)=(-z_o^*,-z^*)
$
and  
\begin{align*}
\begin{split}
\norm{\mathbf I -\derD P^\sharp_\mu(-z_o,-z) }& =\norm{I- (I + \mu^2 \nabla^2 f^\sharp( P^\sharp_\mu(-z_o,-z)))^{-1}}\\
&=\norm{(I + \mu^2 \nabla^2 f^\sharp(P^\sharp_\mu(-z_o,-z)))^{-1} ( \mu^2 \nabla^2 f^\sharp(P^\sharp_\mu(-z_o,-z))) }= O(\mu).
\end{split}
\end{align*}
Hence 
$\norm{\derD P_\mu(z_o,z)-\mathbf 0}=\norm{\mathbf I -\derD P^\sharp_\mu(-z_o,-z)}= O(\mu)= O(\norm{(z_o-zo^*,z-z^*,\mu)}).$ 
We now verify expression \eqref{eq:checkKmn} for $K_{m,n}^\sharp$. By Moreau decomposition 
$(z_o^*,z^*)=\Pi_{K^\sharp}(z_o^*,z^*) - \Pi_{K}(-z_o^*,-z^*),
$
we imply that the projector onto $K^\sharp_{m,n}$ is differentiable at $(z_o^*,z^*)$ if and only if projector onto $K_{m,n}$ is differentiable at $(-z_o^*,-z^*)$. On the other hand, by the result for $K_{m,n}$, we have 
$$\norm{\derD P_\mu(-z_o,-z) - \derD \Pi_{K_{m,n}}(-z_o^*,-z^*)} = O(\norm{(z_o-z_o^*,z-z^*,\mu)}).$$ Therefore, the result follows from $\derD P_\mu^\sharp(z_o,z)=\mathbf I-\derD P_\mu(-z_o,-z)$. \QEDB
\end{proof}

\subsubsection{The equivalence of the differentiability of the projection and  the strict complementarity}
We now prove the remaining part of Theorem \ref{theorem:verify} that is the equivalence of the differentiability of the projection at $(z_o^*,z^*)=(x_o^*,x^*)-(y_o^*,y^*)$ and the strict complementarity of $((x_o^*,x^*),(y_o^*,y^*))$. Here $((x_0^*,x^*),(y_o^*,y^*))$ is a pair of the solutions of the VI.
\begin{proof}
The cases $(x_0^*,x^*) = 0$ or $(x_0^*,x^*)\in \inter(K_{m,n})$ are trivial. We consider non-trivial case, i.e.,$(x_0^*,x^*) \ne 0$ and $(x_0^*,x^*)\not\in \inter(K_{m,n})$.  From  \cite[Section 6.3]{Chua_Hien} we have
\[x^*=u^* [\Diag(\sigma_1^*,\ldots,\sigma_m^*)\quad  0] (v^*)^T, y^*=u^*[\Diag(\tau_1^*,\ldots,\tau_m^*)\quad  0](v^*)^T,
\]
where $x_0^*=\sigma_1^*=\ldots=\sigma_r^* > \sigma_{r+1}^*\geq \ldots\geq \sigma_{m+1}^*=0$
and 
$\tau_1^*\leq \ldots\leq \tau_{r^\sharp}^*<\tau^*_{r^\sharp+1}=\ldots=\tau^*_{m+1}=0,y_0^*=-\sum\tau_i^* 
$
for some $r,r^\sharp \in \{ 1,\ldots,m\}, r\geq r^\sharp$. 
By \cite[Example 5.7]{deSa}, we have 
\[\mathcal{F}_{K}=\left\{(x_0,x): x= u^* \begin{pmatrix}
x_0 I_r & 0\\ 0 & M
\end{pmatrix} (v^*)^T, M\in R^{(m-r)\times (n-r)}, \norm{M}\leq x_0 \right\}
\] is a face of $K_{m,n}$ containing $(x_o^*,x^*)$. This face is with respect to the standard face $$ S_r^\infty =\left\{ (x_0,\bar{x})\in C_n: \bar{x}_i=x_0 \,\text{for}\, 1\leq i\leq r\right\}.$$
By \cite[Theorem 6.2]{deSa}, $(x_0^*,x^*) \in \relint (\mathcal{F}_{K})$. 
Similarly, by \cite[Example 5.6]{deSa}, we have  
\[\mathcal{F}_{K^\sharp}=\left\{(y_0,y):y=u^*\begin{pmatrix}
-N & 0\\
0 & 0
\end{pmatrix} (v^*)^T, N\in S_+^{r^\sharp}, \trace N=y_0\right\}
\] 
is a face of $K^\sharp_{m,n}$ containing $(y_0^*,y^*)$. This face is with respect to the standard face 
\[S_{r^\sharp}^1=\left\{ (y_0,\bar{y})\in C_n^\sharp: \sum_{i=1}^{r^\sharp} \bar{y}_i=y_0\right\}.
\]
Furthermore, by \cite[Theorem 6.2]{deSa}, we have $(y_0^*,y^*) \in \relint (\mathcal{F}_{K^\sharp})$. 
Therefore, if $(x_0^*,x^*)$ and $(y_0^*,y^*)$ are strictly complementary then $\mathcal{F}_{K^\sharp}=\mathcal{F}_{K}^\bigtriangleup$. Moreover, we note that 
\[\mathcal{F}_{K}^\bigtriangleup=\left\{(y_0,y): y=u^*\begin{pmatrix}
-N & A \\ C & 0
\end{pmatrix} (v^*)^T, N \in R^{r\times r} ,\trace (N)=y_0\geq \left\|\begin{pmatrix}
N & -A\\-C & 0
\end{pmatrix} \right\|_*\right\}.
\]
 This implies that $r=r^\sharp$; otherwise, the point $(y_0,\tilde{y})$, which is defined by  
\begin{equation}
 \label{eq:face}
 y_0>0,\tilde{y}=u^*\begin{pmatrix}
-Diag(\tilde{y}_1,\ldots,\tilde{y}_r) & 0\\
0 & 0
\end{pmatrix} (v^*)^T,\sum_{i=1}^r\tilde{y}_i=y_0,\tilde{y}_i>0,
 \end{equation}
 belongs to $\mathcal{F}_{K}^\bigtriangleup$ but does not belong to $\mathcal{F}_{K^\sharp}$, this gives a contradiction. Then we deduce that 
$$z^*=x^*-y^*=u^*[\Diag( x_0^*-\tau_1^*,\ldots,x_0^*-\tau_r^*,\sigma_{r+1}^*,\ldots,\sigma_m^*)\quad 0] (v^*)^T.$$
Therefore the projector of $z^*$ onto $K_{m,n}$ is differentiable by Theorem \ref{theorem:diff_proj_oper_norm}.

Conversely, suppose that  the projector onto $K_{m,n}$ is differentiable at $(z_o^*,z^*)$, then we have $r=r^\sharp$. We know that each face of $K_{m,n}^\sharp$  unique determines a standard face of $C_n^\sharp$. Suppose that  $\mathcal{F}_{K^\sharp}\ne \mathcal{F}_{K}^\bigtriangleup$, i.e., $S_{r^\sharp}^1$ is not the standard face of   $\mathcal{F}_{K}^\bigtriangleup$. Then the standard face of $\mathcal{F}_{K}^\bigtriangleup$ has the form 
$S_{\bar{r}}^1=\left\{ (y_0,\bar{y})\in C_n^\sharp: \sum_{i=1}^{\bar{r}} \bar{y}_i=y_0\right\}
$
with $\bar{r}\ne r$. If $\bar{r}<r$ then the point $(y_0,\tilde{y})$ defined in \eqref{eq:face} belongs to $\mathcal{F}_{K}^\bigtriangleup$ but definitely does not belong to the face of $K_{m,n}^\sharp$ generated by $S_{\bar{r}}^1$. This is a contradiction. If $\bar{r}>r$ then the point  $(\bar{y}_0,\bar{y})$ with $\bar{y}_0>0$ and  
\[\bar{y}=u^*\begin{pmatrix}
-Diag(\bar{y}_1,\ldots,\bar{y}_{\bar{r}}) & 0\\
0 & 0
\end{pmatrix} (v^*)^T,\sum_{i=1}^{\bar{r}}\bar{y}_i=y_0,\bar{y}_i>0,
\]
belongs to the face  of $K_{m,n}^\sharp$ generated by $S_{\bar{r}}^1$ but definitely does not belong to  $\mathcal{F}_{K}^\bigtriangleup$. We again get a contradiction. Therefore,  $\mathcal{F}_{K^\sharp}=\mathcal{F}_{K}^\bigtriangleup$, i.e., $((x_o^*,x^*),(y_o^*,y^*))$ is strict complementary. ~ \QEDB
\end{proof}
\section{Conclusion} 
\label{sec:end}
We analyse an inexact non-interior continuation method for variational inequalities over general closed convex sets. The method can deal with large scale problems by solving involving Newton equations inexactly. Proposition \ref{prop:bsa_property} is the key to achieving the global linear convergence of the algorithm. A $\vartheta$-self-concordant barrier of $X$ is the sufficient condition to get the inequality $\norm{\derD^2 p_\mu(z)} \leq \cfrac{1}{4\mu}$ in this proposition. For local convergence, Theorem \ref{theorem:seq_Jacob_conver} serves as a cornerstone to establish the local quadratic convergence of the algorithm. Therefore, in Section \ref{sec:application}, we always choose self-concordant barriers in application the algorithm to concrete closed convex sets, and verify the condition $\norm{\derD p_\mu(z)-\derD \Pi_X(z^*)}=O(\norm{(z-z^*,\mu)})$ for these sets. We further prove that differentiability of $\Pi_X$ at $z^*$ is equivalent to strict complementarity of $(x^*,y^*)$ when $X$ is a non-negative orthant, a semidefinite cone, an epigraph of matrix operator norm or an epigraph of matrix nuclear norm. 

\section*{Acknowledgement} We thank the anonymous reviewers for their meticulous and insightful comments, which help us improve the paper. LTKH gives special thanks to Prof. Nicolas Gillis for his support.
\bibliographystyle{plain}
\bibliography{NICM}
\begin{appendix}

\small
\section{Technical proofs}
\subsection{Proof of Proposition \ref{prop:Newton_direction_bounded}}
\label{proof1}
(i) From Equation \eqref{eq:Newton_direction} we get
$$\norm{\triangle \tilde{w}^{(k)}} \leq C ( \norm{H_{\mu_k}(w^{(k)})} + \norm{r_1^{(k)}}) \leq C(1+\theta_1) \Psi_{\mu_k}(w^{(k)}) \leq C(1+\theta_1)\beta \mu_k. $$
(ii) We have
\begin{align} \label{temp3}
\begin{split}
\norm{ H_0(w^{(k)})} - \Psi_{\mu_k}(w^{(k)})& \leq  \norm{\phi_0(w^{(k)})}  - \norm{\phi_{\mu_k}(w^{(k)})}\\
& \leq \norm{\phi_0(w^{(k)}) - \phi_{\mu_k}(w^{(k)})}\\
& = \norm{p_0(x^{(k)}-y^{(k)}) - p_{\mu_k}(x^{(k)}-y^{(k)})} \leq \sqrt{\vartheta} \mu_k,
\end{split}
\end{align}
where we have used the property $\norm{(a,b)} \leq \norm{a} + \norm{b}$ for the first inequality and Theorem \ref{theorem:bsa} for the last inequality. Inequality \eqref{temp3} with Equation \eqref{eq:pure_Newton_direction} give us
\begin{align*} \|\triangle \hat{w}\|& \leq C (\norm{ H_0(w^{(k)})} + \norm{r_2^{(k)}}) \leq C(1+\theta_2)( \Psi_{\mu_k}(w^{(k)}) + \sqrt{\vartheta}\mu_k )\\
 &\leq C(1+\theta_2)(\beta + \sqrt{\vartheta}) \mu_k.
\end{align*}
\subsection{Proof of Theorem \ref{theorem:verify}}
\label{proof2}
\subsubsection{Non-negative orthant $\mathbb{R}^n_+$} Gradient and Hessian of the barrier function are
$$ \nabla f(x)= -\sum\limits_{i=1}^n\cfrac{1}{x_i} e_i, \quad \nabla ^2 f(x) = \sum\limits_{i=1}^n \cfrac{1}{x_i^2} e_i e_i^T,
$$
where $e_i$ denote the $i-$th standard unit vector of $\mathbb{R}^n$. 
The corresponding barrier-based smoothing approximation is $p_\mu(z)=\cfrac12\sum\limits_{i=1}^n\bracket{z_i+ \sqrt{z_i^2 + 4\mu^2}} e_i$.
Its Jacobian is $$\der p_\mu(z)= \cfrac12\Diag\left(\left. 1 + \cfrac{z_i}{\sqrt{z_i^2 + 4\mu^2}}\right|_{i=1,\ldots,n}\right).$$ 
The projection of $z$ onto $\mathbb{R}^n_+$ is $\Pi_{\mathbb{R}^n_+}(z)=[z]_+$. We observe that the projector is differentiable at $z^*$ if and only if $z_i^*\ne 0, \forall i=1,\ldots,n$. On the other hand, a pair $(x^*,y^*)$ is strictly complimentary if and only if $x_i^*+y_i^* >0$ for all $i=1,\ldots, n$, see \cite{Bintong}. Furthermore, $x^*+y^*=\Pi_{\mathbb{R}^n_+}(z^*)+\Pi_{\mathbb{R}^n_+}(-z^*)$. Hence, it is easy to see that differentiability of the projector at $z^*$ is equivalent to strict complimentarity of $(x^*,y^*)$. 

Now let $z_i^*\ne 0$ for $i=1,\ldots,n$, then we observe that the Jacobian $\der p_\mu(z)$ converges to $$\der \Pi_{\mathbb{R}^n_+}(z^*)=\cfrac12\Diag\left(\left. 1 + \cfrac{z_i^*}{|z_i^*|}\right|_{i=1,\ldots,n}\right)$$ when $(z,\mu)\rightarrow (z^*,0)$.  Since the map $(z,\mu)\mapsto\der p_\mu(z)$ is continuously differentiable at $(z^*,0)$, thus is locally Lipschitz at this point.  Consequently, $\norm{\der p_\mu(z)-T^*}= O(\norm{(z-z^*,\mu)})$. On the other hand, $\norm{\derD p_\mu(z)-\derD \Pi_{\mathbb{R}^n_+}(z^*)}=\norm{\der p_\mu(z)-\der \Pi_{\mathbb{R}^n_+}(z^*)}$. Thus, we get the result. 
\subsubsection{Positive semidefinite cone $\mathbb{S}_+^n$}
We have $
\nabla f(x)= -x^{-1}.$
From the equation $x + \mu^2 \nabla f(x) = z,
$ we deduce that the corresponding barrier-based smoothing approximation is
$$ p_\mu(z)=\cfrac12\left(z+ (z^2 + 4\mu^2 I)^{1/2}\right).
$$
Denote
 $\mathfrak{g}: u\in \mathbb{R} \mapsto \mathfrak{g}(u)= u+ \sqrt{u^2 + 4\mu^2}$, $\mathfrak{g}'(u)=1+ \cfrac{u}{\sqrt{u^2 + 4\mu^2}}$ and $\mathfrak{g}^{(1)}$ is a matrix whose $(i,j)$-th entry  with respect to a vector $d$ is 
\begin{align*}
(\mathfrak{g}^{(1)}(d))_{ij}&=\left\{ \begin{array}{ll}
\cfrac{\mathfrak{g}(d_i)-\mathfrak{g}(d_j)}{d_i - d_j}  &\text{if} \; d_i\ne d_j\\
\mathfrak{g}'(d_i) & \text{if}\; d_i = d_j
\end{array} \right.\\
& = \left\{ \begin{array}{ll}1+ \cfrac{\sqrt{d_i^2 + 4\mu^2} - \sqrt{d_j^2 + 4\mu^2}}{d_i - d_j}  &\text{if} \; d_i\ne d_j\\
1+ \cfrac{d_i}{\sqrt{d_i^2 + 4\mu^2}} & \text{if}\; d_i = d_j
\end{array}\right.\\
& = 1+\cfrac{d_i + d_j}{\sqrt{d_i^2 + 4\mu^2} + \sqrt{d_j^2 + 4\mu^2}}.
\end{align*}

Let $z=q \Diag(\lambda_f(z)) q^T$, then
$\derD p_\mu(z)[h]= \cfrac12 q \left[ \mathfrak{g}^{(1)} (\lambda_f(z)) \circ (q^T h q)\right] q^T.
$
Projection of $z$ onto $\mathbb{S}^n_+$ is $\Pi_{\mathbb{S}^n_+}(z)=q\Diag\bracket{[\lambda_f(z)]_+}q^T$.  We see that $\Pi_{\mathbb{S}^n_+}(\cdot)$ is differentiable at $z^*$ if and only if all eigenvalues $\lambda_i^*$, for $ i=1,\ldots,n$, of $z^*$ are non-zeroes. Furthermore, strict complementarity of $(x^*,y^*)$ is equivalent to the condition that all eigenvalues of $x^*+y^*$ is positive. We now let $z^*=\hat q\Diag(\lambda_f(z^*)) \hat q^T $ be the eigenvalue decomposition of $z^*$. Then, 
\begin{align*}
x^*+y^*&=\Pi_{\mathbb{S}^n_+}(z^*)+ \Pi_{\mathbb{S}^n_+}(-z*)\\
&=\hat q\Diag\bracket{[\lambda_f(z^*)]_+} \hat q^T+ \hat q\Diag\bracket{[-\lambda_f(z^*)]_+} \hat q^T\\
&=\hat q \Diag\bracket{[\lambda_f(z^*)]_+ + [-\lambda_f(z^*)]_+}\hat q^T.
\end{align*}
Hence differentiability of $\Pi_{\mathbb{S}^n_+}(\cdot)$ at $z^*$ is equivalent to strict complementarity of $(x^*,y^*)$.

Now we consider $z^*$ whose eigenvalues are non-zeros. Let $(z,\mu)$ go to $(z^*,0)$, then $\lambda_f(z)$ converges to $\lambda^*$. Let $\bar{q}$ be a limit point of $q$. We then have $z^*=\bar{q} \Diag(\lambda^*) \bar{q}^T$ with $\bar{q} \in \mathcal{O}^n(z^*).
$
We deduce from $\lambda_i^* \ne 0$, $i=1,\ldots,n$ that 
$$(\mathfrak{g}^{(1)} (\lambda_f(z)))_{ij}= 1 + \cfrac{\lambda_i + \lambda_j}{\sqrt{\lambda_i^2 + 4\mu^2} + \sqrt{\lambda_j^2 + 4\mu^2}}\to 1+ \cfrac{\lambda_i^* + \lambda_j^*}{|\lambda_i^*| + | \lambda_j^*|}.$$  
Therefore, by Theorem \ref{theorem:seq_Jacob_conver}, when $(z,\mu)\to (z^*,0)$, where  $z^*$ are differential points of $\Pi_{\mathbb{S}^n_+}(\cdot)$, $\derD p_\mu(z)$ converges to $\derD \Pi_{\mathbb{S}^n_+}(z^*)$ with 
\begin{align}\label{eq:der_proj_Sn}
\derD \Pi_{\mathbb{S}^n_+}(z^*)[h]=\cfrac12 \bar{q} \left[ \bar{\mathfrak{g}}^{(1)} (\lambda^*) \circ (\bar{q}^T h \bar{q})\right] \bar{q}^T,
\end{align}
where $\bar{\mathfrak{g}}^{(1)}(\lambda^*)_{ij}= 1+ \cfrac{\lambda_i^* + \lambda_j^*}{|\lambda_i^*| + | \lambda_j^*|}$. Note that Formula \eqref{eq:der_proj_Sn} is independent of the choice $\bar{q}$. Finally, similarly to the case $\mathbb{R}_+^n$, we have $\norm{\derD p_\mu(z) - \derD \Pi_{\mathbb{S}^n_+}(z^*)}=O(\norm{z-z^*,\mu})$ since $(z,\mu)\mapsto\derD p_\mu(z)$ is locally Lipschitz around $(z^*,0)$.

\subsection{Proof of Proposition \ref{prop:limit_polyhedral}}
\label{proof4}
We note that $\mathcal{F}_{I^*}$  is the  unique neighbour face of $z^*$ and $z^* \in \inter (\mathcal{F}_I + \mathcal{N}_I)$ as the projector is differentiable at $z^*$ (see Proposition \ref{prop:polyhedral_properties}). When $(z,\mu)\rightarrow (z^*,0)$, we have $x\rightarrow \Pi_K(z^*)=\bar{z}^*,$ satisfying $A_i\bar{z}^* =b_i, \forall i\in \mathcal{I}^*$ and  $A_i\bar{z}^* > b_i, \forall i\not\in \mathcal{I}^*.$ 
From Equation \eqref{eq:smooth_equa_polyhedral} we have
\begin{align} \label{temp5}
 z^* -\bar{z}^* = \lim\limits_{(z,\mu)\rightarrow (z^*,0)} (x-z) =   \lim\limits_{(z,\mu)\rightarrow (z^*,0)} \sum\limits_{i\in \mathcal{I}^*} \cfrac{\mu^2}{A_i x-b_i} ( -A_i^T)
\end{align}
If there exists $j\in \mathcal{I}^*$ and a subsequence $(z,\mu)_k\rightarrow (z^*,0)$ such that  $\cfrac{\mu_k^2}{A_j x_k - b_j} \rightarrow 0$  then we take the limit of this subsequence in \eqref{temp5} to get
$$z^* -\bar{z}^* = \lim\limits_{k\rightarrow \infty}  \sum\limits_{i\in \mathcal{I}^*\setminus\{j\}} \cfrac{\mu_k^2}{A_i x_k-b_i} ( -A_i^T) \in \coni \{ -A_i^T : i\in  \mathcal{I}^*\setminus\{j\}\} = \mathcal{N}_{ \mathcal{I}^*\setminus\{j\}}  $$
On the other hand, $\bar{z}^*\in \mathcal{F}_{\mathcal{I}^*\setminus\{j\}}$ as $A_i\bar{z}^*=b_i \;\forall i\in \mathcal{I}^*\setminus\{j\}$. Hence $z^* = \bar{z}^* + z^* -\bar{z}^* \in \mathcal{F}_{I^*\setminus\{j\}} + \mathcal{N}_{\mathcal{I}^*\setminus\{j\}}$, which implies $\mathcal{I}^*\setminus\{j\}$ is a neighbour face of $z^*$. This contradicts to the fact $\mathcal{I}^*$ is the unique neighbour face of $z^*$. Therefore, for all $i\in \mathcal{I}^*, \cfrac{\mu^2}{A_i x -b_i}$ only have nonzero limit points. The result follows then.
\subsection{Proof of Proposition \ref{prop:proj_vector_cone}}
\label{proof3}
By re-indexing $z^*$ if necessary, we can assume that $\pi$ is the identity permutation. For each $i\in \{1,\ldots,n \}$, if $|x_i^*| < t^*$ then 
$$x_i^* = \lim x_i = \lim (z_i-\cfrac{2\mu^2}{t^2-x_i^2}x_i) = z_i^*.
$$
Together with $\sgn(x_i)=\sgn(z_i)$ and $t> |x_i|$, we deduce that $x_i^*=\sgn(z_i^*) t^* \; \text{or} \; z_i^* \; \text{for}\; i=1,\ldots,n.
$
Moreover, $|x_i| < \min \{t,|z_i| \}$ further implies that 
$$x_i^*=\begin{cases} \sgn(z_i^*) t^* &\mbox{if} \; t^* < |z_i^*|\\
z_i^* &\mbox{if} \; t^* > |z_i^*|.\end{cases}
$$
Thus, there exists a unique positive integer $\mathbf k^*$ such that 
$$x_i^* =\begin{cases}\sgn(z_i^*) t^* &\mbox{for}\; i=1,\ldots, \mathbf k^*,\\
z_i^* &\mbox{for} \; i=\mathbf k^*+1,\ldots,n,  
\end{cases} 
$$
and $|z_{\mathbf k^*}^* |> t^* \geq |z_{\mathbf k^*+1}^*|$. Summing up $(n+1)$ equations in \eqref{eq:smooth_equa_vector} gives 
$$z_o + \sum\limits_{i=1}^n|z_i| = t+ \sum\limits_{i=1}^n|x_i| - \mu^2 \sum\limits_{i=1}^n \cfrac{2(t-|x_i|)}{t^2 - x_i^2} = t+ \sum\limits_{i=1}^n |x_i| - \mu^2 \sum\limits_{i=1}^n \cfrac{2}{t+ |x_i|}.
$$
If $t^*>0$ then taking limnit gives, $z_o^*+\sum\limits_{i=1}^n |z_i^*| = t^* + \sum\limits_{i=1}^n|x_i^*|$, and hence
$$(\mathbf k^*+1)t^* = t^* + \sum\limits_{i=1}^{\mathbf k^*}|x_i^*| = z_o^* + \sum\limits_{i=1}^n|z_i^*| - \sum\limits_{i=1}^n|x_i^*|=z_o^* + \sum\limits_{i=1}^{\mathbf k^*}|z_i^*|.
$$
If $t^*=0$ then $|x_i|<t$ implies that $x_i^*=0$ for $i=1,\ldots,n$. Thus 
\begin{align*}
z_o^* + \sum\limits_{i=1}^n|z_i^*| &= \lim (z_o+\sum\limits_{i=1}^n|z_i|) = \lim ( t+ \sum\limits_{i=1}^n|x_i| - \mu^2\sum\limits_{i=1}^n\cfrac{2}{t+|x_i|}\\
& \leq \lim (t+\sum\limits_{i=1}^n|x_i|) = t^* + \sum\limits_{i=1}^n|x_i^*|=0.
\end{align*}
Subsequently, $z_o^* + \sum\limits_{i=1}^{\mathbf k^*}|z_i^*| \leq z_o^* + \sum\limits_{i=1}^n|z_i^*|\leq 0$. Hence, $$t^* = \max \left\{\cfrac{1}{\mathbf k^*+1}(z_o^* + \sum\limits_{i=1}^{\mathbf k^*} | z^*_{\pi(i)}|),0 \right\}.$$
\section{Example} 
\label{example} 
We consider the second order cone in $\mathbb{R}^3$$$K_2=\{ (t,z): z \in \mathbb{R}^2, t\in \mathbb{R}_+,  \norm{z} \leq t\}.
$$ 
Firstly, we use the barrier $f^{(1)}(t,z)=-\log(t^2-\|z\|^2).$
Denote $\mathbf M=t^2 - z_1^2 - z_2^2$. The gradient of $f^{(1)}$ is
$$ \nabla f^{(1)}(t,z)=\bracket{-2t \mathbf M^{-1}, 2z_1 \mathbf M^{-1}, 2z_2 \mathbf M^{-1}}.$$
The smoothing approximation $p^{(1)}_\mu(t^o,z^o)=(t,z)$ regarding to $f^{(1)}(t,z)$ is computed by
\begin{equation} \label{temp03}
\begin{cases} t -\mu^2 (2 t \mathbf M^{-1})&=t^o\\
 z_1 + \mu^2 (2 z_1 \mathbf M^{-1}) &= z_1^o\\
 z_2 + \mu^2 (2 z_2 \mathbf M^{-1})&=z_2^o.
\end{cases}
\end{equation}
 The unique solution of  \eqref{temp03} is
$$\begin{cases}
t=\cfrac14\bracket{2t^o+\sqrt{(t^o-\norm{z^o})^2 + 8\mu^2} + \sqrt{(t^o+\norm{z^o})^2 + 8\mu^2}}\\
z_1=\cfrac14 \cfrac{z_1^o}{\norm{z^o}} \bracket{2\norm{z^o} +  \sqrt{(t^o+\norm{z^o})^2 + 8\mu^2} - \sqrt{(t^o-\norm{z^o})^2 + 8\mu^2} }\\
z_2=\cfrac14 \cfrac{z_2^o}{\norm{z^o}} \bracket{2\norm{z^o} +  \sqrt{(t^o+\norm{z^o})^2 + 8\mu^2} - \sqrt{(t^o-\norm{z^o})^2 + 8\mu^2} }.
\end{cases}
$$
Denote $s_1=\sqrt{(t^o-\norm{z^o})^2 + 8\mu^2}, s_2=\sqrt{(t^o+\norm{z^o})^2 + 8\mu^2}$. We have
$$\der p_\mu^{(1)}(t^o,z_1^o,z_2^o)=\begin{pmatrix}
d_{11} & d_{12} & d_{13}\\
d_{21} & d_{22} & d_{23} \\
d_{31} & d_{32} & d_{33},
\end{pmatrix} $$ where 
\begin{align*}
d_{11}&=\frac12+\frac{t^o(t^o-\norm{z^o})}{4s_1}+\frac{t^o(t^o+\norm{z^o})}{4s_2}, \\
d_{12}=d_{21}&=-\frac{(t^o-\norm{z^o})z_1^o}{4\norm{z^o}s_1}+\frac{(t^o+\norm{z^o})z_1^o}{4\norm{z^o}s_2},\\
d_{13}=d_{31}&=-\frac{(t^o-\norm{z^o})z_2^o}{4\norm{z^o}s_1}+\frac{(t^o+\norm{z^o})z_2^o}{4\norm{z^o}s_2},\\
d_{22}&=\frac12+\frac{(z_1^o)^2}{\norm{z^o}^2} \bracket{\frac{t^o+\norm{z^o}}{s_2} + \frac{t^o-\norm{z^o}}{s_1}} +\frac{(z_2^o)^2(s_2-s_1)}{\norm{z^o}^2},\\
d_{23}=d_{32}&=\frac{z_1^o}{\norm{z^o}}\bracket{\frac{(t^o+\norm{z^o})z_2^o}{s_2} + \frac{(t^o-\norm{z^o})z_2^o}{s_1}},\\
d_{33}&=\frac12+\frac{(z_2^o)^2}{\norm{z^o}^2} \bracket{\frac{t^o+\norm{z^o}}{s_2} + \frac{t^o-\norm{z^o}}{s_1}} +\frac{(z_1^o)^2(s_2-s_1)}{\norm{z^o}^2}.
\end{align*}
We choose $(t^o,z_1^o,z_2^o)$ such that $ (t^o,z_1^o,z_2^o) \to (t^*,z_1^*,z_2^*)=(0,1,0)$; then $s_1\to 1, s_2 \to 1, \norm{z^o}\to 1$. We imply 
$$\lim\limits_{(t^o,z_1^o,z_2^o,\mu)\rightarrow (0,1,0,0)} \der p_\mu^{(1)}(t^o,z_1^o,z_2^o) = \begin{pmatrix} 1/2 & 1/2 & 0\\
1/2 & 1/2 & 0\\
0 & 0 & 1/2
\end{pmatrix},
$$ which equals to $\der\Pi_{K_2}(0,1,0)$ by Theorem \ref{theorem:seq_Jacob_conver}. Now we use another barrier 
$$ f^{(2)} (u,v_1,v_2) = -\log(u^2-\|v\|^2)-\log ( u-v_1) - \log (u+v_1).$$
 Denote $\mathbf M_\mu=u^2 - v^2 - w^2$. Gradient $\nabla f^{(2)}(u,v_1,v_2)$ of $ f^{(2)}$ is 
$$ \bracket{-2u \mathbf M_\mu^{-1} -\cfrac{1}{u-v_1} - \cfrac{1}{u+v_1}, 2v_1 \mathbf M_\mu^{-1} +\cfrac{1}{u-v_1}  - \cfrac{1}{u+v_1} , 2v_2 \mathbf M_\mu^{-1}}.
$$
Let $p_\mu^{(2)}(t^*,z_1^*,z_2^*)=p_\mu^{(2)}(0,1,0)=(u,v_1,v_2)$, which is defined by
$$\begin{cases} u - \mu^2 \bracket{2 u \mathbf M_\mu^{-1} +\cfrac{1}{u-v_1} + \cfrac{1}{u+v_1}}  =0 \\ 
v_1 + \mu^2 \bracket{ 2 v_1 \mathbf M_\mu^{-1} + \cfrac{1}{u-v_1} -\cfrac{1}{u+v_1}} = 1\\
v_2 + \mu^2 (2 v_2 \mathbf M_\mu^{-1}) = 0.
\end{cases}
$$
The third equation implies $v_2=0$, hence $\mathbf M_\mu=u^2 - v_1^2 =(u-v_1)(u+v_1).$ Thus, the first and the second equation imply
$$ \begin{cases}
u+v_1 - \cfrac{4\mu^2}{u+v_1}=1\\
u-v_1 -\cfrac{4\mu^2}{u-v_1}=-1,
\end{cases}
$$
which give $u=\cfrac12\sqrt{1+16\mu^2}, v_1=\cfrac12, \mathbf M_\mu = 4\mu^2$. Denote $a=\sqrt{1+16\mu^2}$. Hessian matrix $\nabla^2  f^{(2)}(u,v_1,v_2)$ of the barrier $ f^{(2)}$ at $(u,v_1,v_2)$ is 
\begin{align*}
\begin{pmatrix}
\frac{1+8\mu^2}{16\mu^4}+\frac{4}{(a-1)^2}+\frac{4}{(a+1)^2} & -\frac{a}{16\mu^4} - \frac{4}{(a-1)^2} +\frac{4}{(a+1)^2} & 0\\
 -\frac{a}{16\mu^4} - \frac{4}{(a-1)^2} +\frac{4}{(a+1)^2} &\frac{1+8\mu^2}{16\mu^4} + \frac{4}{(a-1)^2} +\frac{4}{(a+1)^2} & 0\\
 0&0 & \frac{1}{2\mu^2}
\end{pmatrix}=\begin{pmatrix}
\frac{1+8\mu^2}{8\mu^4}& \frac{-a}{8 \mu^4}&0\\ 
 \frac{-a}{8 \mu^4} &\frac{1+8\mu^2}{8\mu^4} & 0\\
 0 & 0 & \frac{1}{2\mu^2}.
\end{pmatrix}
\end{align*}
We remind that 
$
\der p^{(2)}_\mu(0,1,0)=[I + \mu^2 \nabla^2 f^{(2)}(u,v_1,v_2)]^{-1}=\begin{pmatrix}
1/2 &  \cfrac{1}{2a} &0 \\
\cfrac{1}{2a} & 1/2&0\\
0 & 0 & 2/3
\end{pmatrix}.
$
Thus, $\lim\limits_{\mu\rightarrow 0}\der p^{(2)}_\mu(0,1,0)$ equals 
$\begin{pmatrix}
1/2&1/2&0\\
1/2&1/2&0\\
0&0&2/3
\end{pmatrix}.$
It does not coincide with the Jacobian of $\Pi_{K_2}$ at $(0,1,0)$. This shows that the limit $\lim\limits_{(t^o,z_1^o,z_2^o,\mu)\rightarrow(0,1,0,0)}\derD p^{(2)}_\mu(t^o,z^o)$ does not exist.
\end{appendix}
\end{document}